\documentclass[10pt]{amsart}
\usepackage{amssymb}

\newtheorem{thm}{Theorem}[section]
\newtheorem{lemma}[thm]{Lemma}
\newtheorem{proposition}[thm]{Proposition}

\newtheorem{definition}[thm]{Definition}
\newtheorem{corollary}[thm]{Corollary}

\newcommand{\p}{\mathbb{P}}
\newcommand{\q}{\mathbb{Q}}

\newcommand{\dom}{\mathrm{dom}}
\newcommand{\ran}{\mathrm{ran}}

\newcommand{\add}{\textrm{Add}}

\newcommand{\h}{\mathrm{ht}}
\newcommand{\col}{\mathrm{Col}}
\newcommand{\restrict}{\upharpoonright}

\begin{document}

\title{A Forcing Axiom for a Non-Special Aronszajn Tree}

\author{John Krueger}

\address{John Krueger \\ Department of Mathematics \\ 
	University of North Texas \\
	1155 Union Circle \#311430 \\
	Denton, TX 76203}
\email{jkrueger@unt.edu}



\thanks{This material is based upon work supported by the National Science Foundation under Grant
		No. DMS-1464859.}

\begin{abstract}
	Suppose that $T^*$ is an $\omega_1$-Aronszajn tree with no stationary antichain. 
	We introduce a forcing axiom \textsf{PFA}($T^*$) for proper forcings which preserve 
	these properties of $T^*$. 
	We prove that \textsf{PFA}($T^*$) implies many of the strong consequences of \textsf{PFA}, 
	such as the failure of very weak club guessing, that all of the cardinal characteristics 
	of the continuum are greater than $\omega_1$, and the $P$-ideal dichotomy. 
	On the other hand, \textsf{PFA}($T^*$) implies some of the consequences of diamond principles, 
	such as the existence of Knaster forcings which are not stationarily Knaster.
\end{abstract}

\maketitle

In one of the earliest applications of the method of forcing in set theory, 
Solovay and Tennenbaum \cite{ST} proved the consistency of Suslin's hypothesis, which is the statement 
that there does not exist an $\omega_1$-Suslin tree. 
Baumgartner, Malitz, and Reinhardt \cite{baumgartner} improved this by showing that 
\textsf{MA}$_{\omega_1}$ implies that every $\omega_1$-Aronszajn tree is special, a statement 
which implies Suslin's hypothesis. 
Shelah \cite[Chapter IX]{shelah} proved that Suslin's hypothesis does not imply that all 
$\omega_1$-Aronszajn trees are special by constructing a model in which there are no $\omega_1$-Suslin 
trees but there exists a non-special Aronszajn tree. 
In Shelah's model, the non-special Aronszajn tree satisfies that there are stationarily many 
levels on which the tree is special, and stationary many levels on which there is no 
stationary antichain (recall that an antichain of a tree is stationary if the set of heights of 
its nodes forms a stationary set). 
Schlindwein \cite{chaz} improved this result by constructing a model in which there are no 
$\omega_1$-Suslin trees but there exists an $\omega_1$-Aronszajn tree which has no stationary antichain.

Building on Woodin's work \cite{woodin} on the $\mathbb{S}_{\mathrm{max}}$ forcing extension 
of $L(\mathbb R)$, Larson \cite{larson} produced a maximal model of the existence of a minimal Suslin tree. 
Larson and Todorcevic \cite{larsontodor} proved the consistency of a positive solution 
to Kat\v{e}tov's problem, by creating a model in which 
every compact space whose square is $T_5$ is metrizable. 
Their method involved analyzing the maximal amount of \textsf{MA}$_{\omega_1}$ consistent 
with the existence of a Suslin tree. 
Later Todorcevic \cite{PFAS} introduced a forcing axiom \textsf{PFA}($S$), 
which is a version of \textsf{PFA} but restricted to 
forcings which preserve the Suslinness of a particular coherent $\omega_1$-Suslin tree.

In this article we expand on these results by developing a forcing axiom for proper forcings 
which preserve a particular $\omega_1$-Aronszajn tree with no stationary antichain. 
Let $T^*$ denote an $\omega_1$-Aronszajn tree with no stationary antichain. 
We let \textsf{PFA}($T^*)$ be the statement that for any proper forcing poset $\p$ 
which forces that $T^*$ is still an $\omega_1$-Aronszajn tree with no stationary antichain, 
for any sequence $\langle D_\alpha : \alpha < \omega_1 \rangle$ of dense subsets of $\p$,  
there exists a filter $G$ such that $G \cap D_\alpha \ne \emptyset$ for all $\alpha < \omega_1$.
This forcing axiom is similar in spirit to the previously studied forcing axiom \textsf{PFA}($S$), 
but has some dramatically different consequences such as the non-existence of $S$-spaces. 
Both forcing axioms are useful for understanding the relationship between consequences 
of \textsf{PFA}, properties of $\omega_1$-trees, and the $\omega_1$-chain condition of forcing posets.

The consistency of \textsf{PFA}($T^*$) is established from a supercompact cardinal 
using forcing iteration preservation theorems of Shelah \cite{shelah} and Schlindwein \cite{chaz}. 
We then derive a variety of consequences of \textsf{PFA}($T^*$). 
These consequences include Suslin's hypothesis, the failure of very weak club guessing, that all of the standard 
cardinal characteristic of the continuum are greater than $\omega_1$, and the $P$-ideal dichotomy. 
We also prove that under \textsf{PFA}($T^*$), the forcing poset consisting of finite antichains of 
$T^*$ ordered by reverse inclusion is Knaster but not stationarily Knaster, which is in contrast 
to \textsf{MA}$_{\omega_1}$ which implies that all Knaster forcings are stationarily Knaster. 
The existence of a Knaster forcing which is not stationarily Knaster is also shown to 
follow from $\Diamond^*$. 
Thus, our forcing axiom \textsf{PFA}($T^*$) implies a significant portion of the consequences 
of \textsf{PFA} while also having consequences in common with diamond principles.

I would like to thank Sean Cox and Justin Moore for helpful discussions on the topics 
in this paper.

\section{Forcings which preserve $T^*$}

In this section, we develop the basic ideas concerning forcing posets which preserve a 
non-special Aronszajn tree.

\begin{definition}
	Let $T$ be a tree with height $\omega_1$. 
	A set $A \subseteq T$ is said to be a \emph{stationary antichain} 
	if $A$ is an antichain of $T$ 
	and the set $\{ \h_{T^*}(x) : x \in A \}$ 
	is a stationary subset of $\omega_1$.
	\end{definition}

Observe that any $\omega_1$-Suslin tree has no uncountable antichain, and hence 
no stationary antichain,  
but any special $\omega_1$-Aronszajn tree has a stationary antichain. 
In fact, if an $\omega_1$-tree 
has a special subtree whose nodes meet stationarily many levels, 
then by a pushing down argument it has a stationary antichain.

For the remainder of the paper, whenever we mention $T^*$ we will assume that 
$T^*$ is an $\omega_1$-Aronszajn tree with no stationary antichain.

\begin{definition}
	A forcing poset $\p$ is said to be \emph{$T^*$-preserving} if $\p$ forces that 
	$T^*$ is an $\omega_1$-Aronszajn tree with no stationary antichain.
	\end{definition}

Let $T$ be an $\omega_1$-tree. 
We define a relation $<_{T}^-$ by letting 
$x <_{T}^- y$ if $\h_{T}(x) < \h_{T}(y)$ and 
$x \not <_{T} y$. 
Observe that $x <_{T}^- y$ implies that $x$ and $y$ 
are not comparable in $T$. 
As a word of caution to the reader, the relation 
$<_{T}^-$ is not transitive.

\begin{definition}
	Let $N$ be a countable set such that $\delta := N \cap \omega_1$ is an ordinal. 
	A node $x \in T_\delta$ is said to be \emph{$(N,T)$-generic} if for any 
	set $A \subseteq T$ which is a member of $N$ and any relation $R \in \{ <_{T}, <_{T}^- \}$, 
	if $x \in A$, then there exists $y \ R \ x$ such that $y \in A$.
	\end{definition}

\begin{lemma}
	Let $\theta$ be a regular cardinal with $T \in H(\theta)$, and $N$ a 
	countable elementary substructure of $(H(\theta),\in,T)$. 
	Let $\delta := N \cap \omega_1$. 
	Assume that $x \in T_{\delta}$ is $(N,T)$-generic. 
	Let $B \subseteq T$ be in $N$ and $R \in \{ <_{T}, <_{T}^- \}$. 
	If for all $y \ R \ x$, $y \in B$, then $x \in B$.
	\end{lemma}

\begin{proof}
	Suppose for a contradiction that for all $y \ R \ x$, $y \in B$, but 
	$x \notin B$. 
	Define $A := T \setminus B$. 
	Then $A \in N$ by elementarity, and $x \in A$. 
	Since $x$ is $(N,T)$-generic, there exists $y \ R \ x$ such that $y \in A$. 
	Then $y \notin B$, which is a contradiction.
	\end{proof}

The next lemma basically appears in \cite{chaz}, but under the assumption that 
$T^*$ is an $\omega_1$-Suslin tree.

\begin{lemma}
	Fix a regular cardinal $\theta$ such that $T^* \in H(\theta)$. 
	Let $N$ be a countable elementary substructure of $(H(\theta),\in,T^*)$. 
	Then for all $x \in T^*_{N \cap \omega_1}$, $x$ is $(N,T^*)$-generic.
\end{lemma}

\begin{proof}
	Let $\delta := N \cap \omega_1$. 
	Assume that $x \in T^*_{\delta}$, and we will show that $x$ is $(N,T^*)$-generic. 
	Fix $A \in N$ and $R \in \{ <_{T^*}, <_{T^*}^- \}$. 
	Suppose for a contradiction that $x \in A$, but for all $y \ R \ x$, $y \notin A$. 
	Define $B$ as the set of $z \in T^*$ such that $z \in A$ and 
	for all $y \ R \ z$, $y \notin A$. 
	Then $B \in N$ by elementarity, and $x \in B$. 
	Since $x \in B$, the set $S := \{ \h_{T^*}(z) : z \in B \}$ contains $\delta$.  
	But $\delta \in C$ for any club $C \subseteq \omega_1$ in $N$, so $S$ is stationary 
	by elementarity.
	
	First, assume that $R$ equals $<_{T^*}$. 
	Consider $z_1$ and $z_2$ in $B$, where $\h_{T^*}(z_1) \le \h_{T^*}(z_2)$. 
	If $z_1$ and $z_2$ have the same height, then they are incomparable in $T^*$. 
	If $\h_{T^*}(z_1) < \h_{T^*}(z_2)$, then since $z_1$ and $z_2$ are both in $B$, 
	by the definition of $B$ 
	we must have that $z_1$ is not $<_{T^*}$-below $z_2$. 
	Thus, any two members of $B$ are incomparable, that is, $B$ is an antichain. 
	Since $\{ \h_{T^*}(z) : z \in B \} = S$ is stationary, $B$ is a stationary antichain, 
	which contradicts that $T^*$ has no 
	stationary antichain.
	
	Secondly, assume that $R$ equals $<_{T^*}^-$. 
	Consider $z_1$ and $z_2$ in $B \cap N$. 
	Since $x \in B$, 
	by the definition of $B$ it is not the case that $z_1$ and $z_2$ are 
	$<_{T^*}^-$-below $x$. 
	As they have lower heights than $x$, that means $z_1$ and $z_2$ are both 
	$<_{T^*}$-below $x$. 
	This implies that $z_1$ and $z_2$ are comparable in $T^*$. 
	By elementarity, it follows that any two members of $B$ are comparable in $T^*$, 
	that is, $B$ is a chain. 
	And since $S$ is stationary, $B$ is uncountable. 
	This contradicts that $T^*$ is an $\omega_1$-Aronszajn tree. 
	\end{proof}

While the property of being $T^*$-preserving is very 
natural, in practice it will be helpful 
to have a related property which is more technical 
(see \cite[Chapter IX]{shelah} and \cite{chaz}).

\begin{definition}
	Let $T$ be an $\omega_1$-tree and $\p$ a forcing poset. 
	Let $\theta$ be a regular cardinal such that $T$ and $\p$ are members of $H(\theta)$ 
	and $N$ a countable elementary substructure of $(H(\theta),\in,T,\p)$. 
	A condition $q \in \p$ is said to be 
	\emph{$(N,\p,T)$-generic} if $q$ is $(N,\p)$-generic, and whenever $x \in T_{N \cap \omega_1}$ 
	is $(N,T)$-generic, 
	then $q$ forces that $x$ is $(N[\dot G_\p],T)$-generic.
\end{definition}

Assume that $q$ is $(N,\p)$-generic. 
Then the following property is easily seen to be equivalent to $q$ being $(N,\p,T)$-generic: 
for any node $x \in T_{N \cap \omega_1}$ which is $(N,T)$-generic, 
$\p$-name $\dot A \in N$ for a subset of $T$, and relation $R \in \{ <_{T}, <_{T}^- \}$, 
whenever $r \le q$ forces that $x \in \dot A$, then there exist $y \ R \ x$ and $s \le r$ such that 
$s$ forces that $y \in \dot A$.

\begin{definition}
	Let $T$ be an $\omega_1$-tree and $\p$ a forcing poset. 
	We say that $\p$ is \emph{$T$-proper} if for all large enough regular cardinals $\theta$ 
	with $\p$ and $T$ in $H(\theta)$, 
	there are club many $N$ in $P_{\omega_1}(H(\theta))$ such that $N$ is an elementary substructure 
	of $(H(\theta),\in,T,\p)$ and for all $p \in N \cap \p$, 
	there is $q \le p$ which is $(N,\p,T)$-generic.
	\end{definition}
	
\begin{lemma}
	Let $\p$ be a proper forcing poset. 
	If $\p$ is $T^*$-preserving, then 
	$\p$ is $T^*$-proper.
	\end{lemma}

\begin{proof}
	Assume that $\p$ is $T^*$-preserving. 
	Let $\theta$ be a large enough regular cardinal such that $T^*$ and $\p$ are in $H(\theta)$, 
	and let $N$ be a countable elementary substructure of $(H(\theta),\in,T^*,\p)$ such that every 
	member of $N \cap \p$ has an $(N,\p)$-generic extension. 
	Let $\delta := N \cap \omega_1$. 
	Consider $p \in N \cap \p$. 
	Fix $q \le p$ which is $(N,\p)$-generic. 
	We claim that $q$ is $(N,\p,T^*)$-generic. 

	Consider $x \in T^*_{\delta}$. 
	Let $G$ be a generic filter on $\p$ which contains $q$, and we claim that 
	$x$ is $(N[G],T^*)$-generic. 
	So $N[G]$ is an elementary substructure of $H(\theta)^{V[G]}$ 
	which contains $T^*$.  
	Since $q$ is $(N,\p)$-generic, $N[G] \cap \omega_1 = N \cap \omega_1 = \delta$. 
	Thus, $x \in T^*_{N[G] \cap \omega_1}$. 
	As $\p$ is $T^*$-preserving, in $V[G]$ 
	we have that $T^*$ is still an $\omega_1$-Aronszajn tree 
	with no stationary antichain. 
	So we can apply Lemma 1.5 in $V[G]$ to conclude that $x$ is $(N[G],T^*)$-generic.  
\end{proof}

While the converse of Lemma 1.8 
is not true in general, we will show below that if $\p$ is $T^*$-proper, 
then there is a $\p$-name $\dot \q$ for a forcing poset such that $\p * \dot \q$ is 
proper and $T^*$-preserving.

The next lemma is proved in \cite[Lemma 14]{chaz}.

\begin{lemma}
	If $\p$ is $T^*$-proper, then $\p$ forces that $T^*$ is an $\omega_1$-Aronszajn tree.
	\end{lemma}

Throughout the paper, we will prove that a multitude of forcing posets are $T^*$-proper. 
As a warm-up and because we will need it in the next section, 
let us prove that every $\omega_1$-closed forcing is $T^*$-proper. 

\begin{definition}
	A forcing poset $\p$ is \emph{$\omega_1$-generically closed} if for all large 
	enough regular cardinals $\theta$, there are club many $N \in P_{\omega_1}(H(\theta))$ 
	such that whenever $\langle p_n : n < \omega \rangle$ is a descending sequence of 
	conditions in $N \cap \p$ which meets every dense open subset of $\p$ in $N$, 
	then this sequence has a lower bound.\footnote{See \cite[Chapter V Section 1]{shelah} 
	for more information about this property.}
	\end{definition}

A sequence as described in Definition 1.10 will be called an \emph{$(N,\p)$-generic sequence}. 
Observe that countably closed forcings are $\omega_1$-generically closed.

\begin{proposition}
	Suppose that $\p$ is $\omega_1$-generically closed. 
	Then $\p$ is $T^*$-proper.
	\end{proposition}

\begin{proof}
	Fix a large enough regular cardinal $\theta$ and a countable elementary 
	substructure $N$ of $H(\theta)$ which contains $\p$ and $T^*$ and satisfies that 
	every descending $(N,\p)$-generic sequence has a lower bound.
	Let $\delta := N \cap \omega_1$.
	
	Let $\langle x_n : n < \omega \rangle$, $\langle \dot A_n : n < \omega \rangle$, 
	and $\langle D_n : n < \omega \rangle$ enumerate all of the nodes 
	in $T^*_{\delta}$, all $\p$-names for subsets of $T^*$ in $N$, and all 
	dense open subsets of $\p$ in $N$. 
	Fix a bijection $f : \omega \to \omega \times \omega \times 3$.
	
	Consider $p \in N \cap \p$. 
	We define by induction a descending sequence of conditions $\langle p_n : n < \omega \rangle$ 
	in $N \cap \p$ as follows. 
	Let $p_0 := p$. 
	Assume that $n < \omega$ and $p_n$ is defined. 
	Let $f(n) = (m,k,j)$. 
	If $j = 0$, fix $p_{n+1} \le p_n$ 
	in $N \cap D_m$ (in this case, we ignore $k$). 
	Suppose that $j \ne 0$. 
	If $j = 1$, let $R$ be the relation $<_{T^*}$, and if $j = 2$, 
	let $R$ be the relation $<_{T^*}^-$.
	
	Let $B_n$ denote the set of $z \in T^*$ such that some extension of $p_{n}$ 
	forces that $z \in \dot A_k$. 
	Then $B_n \in N$ by elementarity. 
	If $x_m \notin B_n$, then let $p_{n+1} := p_{n}$. 
	Suppose that $x_m \in B_n$. 
	Since $x_m$ is $(N,T^*)$-generic, there exists $y \ R \ x_m$ such that 
	$y \in B_n$. 
	Then $y \in N$. 
	By the definition of $B_n$ and elementarity, we can fix $p_{n+1} \le p_n$ 
	in $N \cap \p$ which forces that $y \in \dot A_k$.
	
	This completes the definition of $\langle p_n : n < \omega \rangle$. 
	Observe that this sequence is $(N,\p)$-generic. 
	So we can find $q \in \p$ such that $q \le p_n$ for all $n < \omega$. 
	In particular, $q \le p$. 
	Clearly, $q$ is $(N,\p)$-generic. 
	To see that it is $(N,\p,T^*)$-generic, let $x \in T^*_{\delta}$, 
	$\dot A \in N$ a $\p$-name for a subset of $T^*$, and 
	$R \in \{ <_{T^*}, <_{T^*}^- \}$. 
	Suppose that $r \le q$ and $r$ 
	forces that $x \in \dot A$.
	
	Fix $m$ and $k$ such that $x = x_m$ and $\dot A = \dot A_k$. 
	Let $j := 1$ if $R$ equals $<_{T^*}$ and $j := 2$ if $R$ equals $<_{T^*}^-$. 
	Fix $n$ such that $f(n) = (m,k,j)$. 
	Now $r \le p_n$ and $r$ forces that $x_m \in \dot A_k$. 
	So $r$ witnesses that $x_m \in B_n$. 
	By construction, there exists $y \ R \ x_m$ 
	such that $p_{n+1}$ forces that $y \in \dot A_k$. 
	Since $r \le p_{n+1}$, $r$ forces that $y \in \dot A$.
	\end{proof}

\begin{corollary}
	Any $\omega_1$-closed forcing poset is $T^*$-proper.
	\end{corollary}

\section{A forcing axiom for $T^*$}

We now introduce and prove the consistency of a forcing axiom for $T^*$.

The following forcing poset for adding a Baumgartner club 
will be used in a number of different contexts throughout the paper, 
beginning in this section. 
It is a slight variation of a forcing of Baumgartner 
\cite[p.\ 926]{baumgartnerproper} for adding a club with finite conditions.

\begin{definition}
	Let $A \subseteq \omega_1$. 
	Define $\textrm{CU}(A)$ to be the forcing poset whose conditions are 
	functions $f$ mapping into $\omega_1$ satisfying:
	\begin{enumerate}
		\item $\dom(f)$ is a finite subset of $A$;
		\item if $\alpha < \beta$ are in $\dom(f)$, then $\alpha \le f(\alpha) < \beta$;
	\end{enumerate}
	and ordered by $g \le f$ if $f \subseteq g$.
\end{definition}

It is not difficult to show 
that if $A$ is stationary, then $\textrm{CU}(A)$ preserves $\omega_1$ and 
forces that the set 
$$
\bigcup \{ \dom(f) : f \in \dot G_{\textrm{CU}(A)} \}
$$
is a club subset of $A$.

The next result states that the property of being $T^*$-proper is satisfied by 
certain countable support forcing iterations.

\begin{thm}
	Let $T^*$ be an $\omega_1$-Aronszajn tree with no stationary antichain. 
	Let 
	$$
	\langle \p_i, \dot \q_j : i \le \alpha, \ j < \alpha \rangle
	$$
	be a countable support forcing iteration such that for all $\gamma < \alpha$, 
	$\p_\gamma$ forces either:
	\begin{enumerate}
		\item $\dot \q_\gamma$ is $T^*$-proper, or
		\item there exists an antichain $I \subseteq T^*$ such that 
		$\dot \q_\gamma = \textrm{CU}(\omega_1 \setminus S)$, 
		where $S = \{ \h_{T^*}(z) : z \in I \}$.
		\end{enumerate}
	Then $\p_\alpha$ is $T^*$-proper.
		\end{thm}
	
	This theorem was proven in \cite[Theorem 21]{chaz} under the assumption that 
	$T^*$ is an $\omega_1$-Suslin tree. 
	However, exactly the same proof works assuming instead that $T^*$ is an $\omega_1$-Aronszajn tree 
	with no stationary antichain. 
	In fact, the only place in the proof of 
	\cite[Theorem 21]{chaz} 
	where the Suslin property of the tree is used 
	is in case 2 of the successor case, which uses \cite[Lemma 8]{chaz}. 
	This lemma has the same conclusion as Lemma 1.5 above, but assumes instead that 
	the tree is Suslin. 
	Thus, the proof of Theorem 2.2 is exactly the same as the proof of \cite[Theorem 21]{chaz}, 
	except replacing its application of \cite[Lemma 8]{chaz} by our Lemma 1.5.

\begin{proposition}
	Let $\p$ be a $T^*$-proper forcing poset. 
	Then there exists a $\p$-name $\dot \q$ for a forcing poset such that 
	$\p * \dot \q$ is proper and $T^*$-preserving.
	\end{proposition}

\begin{proof}
	For simplicity, we assume that there exists a regular cardinal $\kappa$ 
	such that $\p * \add(\omega_1) * \add(\omega_2)$ forces that $\kappa = \omega_2$. 
	If there does not, then the general case can be proven by a 
	straightforward argument involving splitting this forcing 
	into a maximal antichain of conditions deciding the value of $\omega_2$. 
	Observe that this three-step forcing iteration forces that $2^\omega = \omega_1$ 
	and $2^{\omega_1} = \omega_2$. 
	
	We define a forcing iteration 
	$\langle \p_i, \dot \q_j : i \le \kappa, \ j < \kappa \rangle$ as follows. 
	The first forcing in the iteration is $\p$, and the second forcing is 
	$\add(\omega_1) * \add(\omega_2)$. 	
	Assume that $2 \le \alpha < \kappa$ and $\p_\alpha$ is defined. 
	We consider a $\p_\alpha$-name $\dot I_\alpha$ for an antichain of $T^*$, and let 
	$\dot \q_\alpha$ be a $\p_\alpha$-name for the forcing poset 
	$\textrm{CU}(\omega_1 \setminus \dot S_\alpha)$, where $\dot S_\alpha$ is a $\p_\alpha$-name 
	for the set $\{ \h_{T^*}(z) : z \in \dot I_\alpha \}$. 
	Observe that after forcing with $\dot \q_\alpha$, $\dot I_\alpha$ is not 
	a stationary antichain. 
	At limit stages $\delta \le \kappa$, we define $\p_\delta$ by taking the limit 
	with countable support.

	This completes the definition. 
	Now $\p$ is $T^*$-proper, and since $\add(\omega_1)$ is forced to be $\omega_1$-closed, 
	it is forced to be $T^*$-proper as well by Corollary 1.12. 
	By Theorem 2.2, it follows that $\p_\kappa$ is $T^*$-proper. 
	Hence, $\p_\kappa$ preserves $\omega_1$, and by Lemma 1.9, $\p_\kappa$ forces that $T^*$ 
	is an $\omega_1$-Aronszajn tree.
	
	Now in $V^{\p_2}$, which models $2^\omega = \omega_1$ and $2^{\omega_1} = \omega_2$, 
	standard arguments show that 
	the tail of the iteration is forcing equivalent to a 
	countable support iteration of length $\omega_2$ 
	of the forcings $\textrm{CU}(\omega_1 \setminus \dot S_\alpha)$, 
	which each have size $\omega_1$. 
	By \cite[Lemma 2.4, Chapter VIII]{shelah}, this iteration is $\omega_2$-c.c. 
	Since $\kappa$ equals $\omega_2$ in $V^{\p_2}$, 
	it follows that $\p_\kappa$ forces that $\kappa = \omega_2$. 
	The chain condition also implies that we can arrange by standard bookkeeping 
	that all antichains of $T^*$ appearing in $V^{\p_\kappa}$ are handled at some stage. 
	It follows that $\p_\kappa$ forces that $T^*$ has no stationary antichain. 
	So in $V^{\p_\kappa}$, $T^*$ is an $\omega_1$-Aronszajn tree with no stationary antichain.

	Let $\dot \q$ be a $\p$-name for the tail of the iteration after forcing with $\p$. 
	Then $\p * \dot \q$ is forcing equivalent to $\p_\kappa$, which is proper and forces 
	that $T^*$ is an $\omega_1$-Aronszajn tree with no stationary antichain. 
	Thus, $\p * \dot \q$ is $T^*$-preserving, completing the proof.
	\end{proof}

\begin{definition}
	Assume that $T^*$ is an $\omega_1$-Aronszajn tree with no stationary antichain. 
	Define $\textsf{PFA}(T^*)$ to be the statement that for any proper $T^*$-preserving forcing poset $\p$, 
	whenever $\langle D_\alpha : \alpha < \omega_1 \rangle$ is a sequence of dense subsets of $\p$, 
	then there exists a filter $G$ on $\p$ 
	such that for all $\alpha < \omega_1$, $G \cap D_\alpha \ne \emptyset$.
	\end{definition}

In practice, we do not have a general way to verify that a proper forcing is $T^*$-preserving. 
Therefore, the following characterization of \textsf{PFA}($T^*$) will be useful.

\begin{proposition}
Assume that $T^*$ is an $\omega_1$-Aronszajn tree with no stationary antichain. 
Then $\textsf{PFA}(T^*)$ is equivalent to the statement that for 
any $T^*$-proper forcing poset $\p$, 
whenever $\langle D_\alpha : \alpha < \omega_1 \rangle$ is a sequence of dense subsets of $\p$, 
then there exists a filter on $G$ on $\p$ 
such that for all $\alpha < \omega_1$, $G \cap D_\alpha \ne \emptyset$.
\end{proposition}

\begin{proof}
	Let \textsf{P} denote the statement which we are claiming to be equivalent to 
	\textsf{PFA}($T^*$). 
	Since any proper $T^*$-preserving forcing poset is $T^*$-proper by Lemma 1.8, 
	\textsf{P} obviously implies \textsf{PFA}($T^*$). 
	Conversely, assume \textsf{PFA}($T^*$). 
	Let $\p$ be a $T^*$-proper forcing poset and 
	$\langle D_\alpha : \alpha < \omega_1 \rangle$ a sequence of dense subsets of $\p$. 
	By Proposition 2.3, there exists a $\p$-name $\dot \q$ such that 
	$\p * \dot \q$ is proper and $T^*$-preserving.
	
	For each $\alpha < \omega_1$, 
	let $E_\alpha$ be the set of conditions $p * \dot q$ 
	in $\p * \dot \q$ such that $p \in D_\alpha$. 
	Then easily each $E_\alpha$ is dense in $\p * \dot \q$. 
	Applying \textsf{PFA}($T^*$), 
	fix a filter $G'$ on $\p * \dot \q$ such that for all $\alpha < \omega_1$, 
	$G' \cap E_\alpha \ne \emptyset$. 
	Define $G := \{ p \in \p : \exists \dot q \ \ p * \dot q \in G' \}$. 
	Then $G$ is a filter on $\p$ and for all $\alpha < \omega_1$, 
	$G \cap D_\alpha \ne \emptyset$.
	\end{proof}

\begin{thm}
	Suppose that $T^*$ is an $\omega_1$-Aronszajn tree with no stationary antichain 
	(for example, if $T^*$ is an $\omega_1$-Suslin tree). 
	Assume that there exists a supercompact cardinal $\kappa$. 
	Then there is a forcing poset which forces that $\kappa = \omega_2$ and 
	$\textsf{PFA}(T^*)$ holds.
	\end{thm}

The proof is nearly identical to the usual construction of a model of \textsf{PFA}. 
We define a forcing iteration of length $\kappa$ with countable support, where at each stage 
we force with either a forcing which is $T^*$-proper and has size less than $\kappa$, 
or with a forcing of the form 
$\textrm{CU}(\omega_1 \setminus S)$, where $S = \{ \h_{T^*}(z) : z \in I \}$ for some 
antichain $I \subseteq T^*$. 
We use a Laver function as a bookkeeping mechanism to anticipate all possible $T^*$-proper 
forcings and all possible antichains of $T^*$ in the final model. 
By Theorem 2.2, such a forcing iteration is $T^*$-proper. 
Also, standard arguments show that it is $\kappa$-c.c., 
and by Lemma 1.9 and the construction, $T^*$ is still an $\omega_1$-Aronszajn tree with no stationary antichain 
after forcing with it. 
Since $\omega_1$-closed forcings are $T^*$-proper, and in particular 
the collapse $\col(\omega_1,\omega_2)$ is $T^*$-proper, the iteration forces that 
$\kappa$ equals $\omega_2$. 
Now the same elementary embedding argument used in the original proof for \textsf{PFA} 
works the same in this context to prove \textsf{PFA}($T^*$).

\section{Strong properness and club guessing}

In this section, we will prove that any strongly 
proper forcing is $T^*$-proper. 
It will follow that \textsf{PFA}($T^*$) implies the failure 
of weak club guessing. 
We then show by a more intricate argument that \textsf{PFA}($T^*$) implies the failure 
of very weak club guessing.

We recall some definitions. 
Let $\p$ be a forcing poset. 
For a set $N$, a condition $q \in \p$ is said to be 
\emph{strongly $(N,\p)$-generic} if 
for any dense subset $D$ of the poset $N \cap \p$, 
$D$ is predense below $q$. 
This property is equivalent to saying that for all $r \le q$, there is $u \in N \cap \p$ 
such that for all $v \le u$ in $N \cap \p$, $r$ and $v$ are compatible. 
We say that $\p$ is \emph{strongly proper} if for all large enough regular cardinals $\theta$ 
with $\p \in H(\theta)$, for club many countable $N \in P_{\omega_1}(H(\theta))$, 
for all $p \in N \cap \p$, there is $q \le p$ which is strongly $(N,\p)$-generic.

\begin{proposition}
	Suppose that $\p$ is a strongly proper forcing poset. 
	Then $\p$ is $T^*$-proper.
\end{proposition}

\begin{proof}
	Fix a large enough regular cardinal $\theta$.  
	Let $N$ be a countable elementary substructure of $(H(\theta),\in,T^*,\p)$ such that 
	every member of $N \cap \p$ has a strongly $(N,\p)$-generic extension. 
	Let $\delta := N \cap \omega_1$. 
	Consider $p \in N \cap \p$. 
	Fix $q \le p$ which is strongly $(N,\p)$-generic. 
	Then in particular, $q$ is $(N,\p)$-generic.
	
	We claim that $q$ is $(N,\p,T^*)$-generic. 
	Let $x \in T^*_{\delta}$, $\dot A \in N$ a $\p$-name 
	for a subset of $T^*$, and $R \in \{ <_{T^*}, <_{T^*}^- \}$. 
	Assume that $r \le q$ and $r$ forces that $x \in \dot A$. 
	We will find $s \le r$ and $y \ R \ x$ such that $s$ forces that $y \in \dot A$.
	
	Since $q$ is strongly $(N,\p)$-generic, we can fix $u \in N \cap \p$ such that 
	for all $v \le u$ in $N \cap \p$, $r$ and $v$ are compatible. 
	Define $B$ as the set of $z \in T^*$ for which there exists a condition below $u$ 
	which forces that $z \in \dot A$. 
	Note that $B \in N$ by elementarity. 
	Since $u$ and $r$ are compatible, and any common extension of them forces that $x \in \dot A$, 
	it follows that $x \in B$. 
	As $x$ is $(N,T^*)$-generic, there exists $y \ R \ x$ such that $y \in B$. 
	Then $y \in N$. 
	By the definition of $B$ and elementarity, we can find $v \le u$ 
	in $N \cap \p$ which forces that $y \in \dot A$. 
	By the choice of $u$, $v$ and $r$ are compatible. 
	Fix $s \le v, r$. 
	Then $s \le r$ and $s$ forces that $y \in \dot A$.
\end{proof}

\begin{corollary}
	$\textsf{PFA}(T^*)$ implies that for any strongly proper forcing poset $\p$, 
	whenever $\langle D_\alpha : \alpha < \omega_1 \rangle$ is a sequence of dense 
	subsets of $\p$, then there exists a filter $G$ on $\p$ such that for all 
	$\alpha < \omega_1$, $G \cap D_\alpha \ne \emptyset$.
	\end{corollary}

Recall that a \emph{ladder system} is a sequence $\vec L = \langle L_\alpha : \alpha \in \lim(\omega_1) \rangle$ 
such that for each $\alpha$, $L_\alpha$ is a cofinal subset of $\alpha$ with order type $\omega$. 
Let $\vec L$ be such a ladder system. 
Consider the following club guessing properties of $\vec L$, which are ordered 
in decreasing strength:
\begin{itemize}
	\item for any club $C \subseteq \omega_1$, there exists $\alpha \in \lim(\omega_1)$ 
	such that $L_\alpha \subseteq C$;
	\item for any club $C \subseteq \omega_1$, there exists $\alpha \in \lim(\omega_1)$ 
	such that $L_\alpha \subseteq^* C$;\footnote{Throughout the paper, $A \subseteq^* B$ will 
	mean that $A \setminus B$ is finite.}
	\item for any club $C \subseteq \omega_1$, there exists $\alpha \in \lim(\omega_1)$ 
	such that $L_\alpha \cap C$ is infinite.
	\end{itemize}
We say that \emph{club guessing}, \emph{weak club guessing}, or \emph{very weak club guessing} 
holds if there exists a ladder system satisfying (1), (2), or (3) respectively.

It is well-known that \textsf{PFA} implies the failure of very weak club guessing 
(\cite[Theorem 3.6]{baumgartnerproper}). 
We will prove that the same holds under \textsf{PFA}($T^*$). 
We use the forcing poset $\textrm{CU}(A)$ described in Section 2.

\begin{lemma}
	For any club $D \subseteq \omega_1$, 
	the forcing poset $\textrm{CU}(D)$ is strongly proper.
	\end{lemma}

\begin{proof}
	Let $\p := \textrm{CU}(D)$. 
	Consider any regular cardinal $\theta$ such that $\p \in H(\theta)$, and let 
	$N$ be a countable elementary substructure of $(H(\theta),\in,\p,D)$. 
	Let $\delta := N \cap \omega_1$. 
	Then $\delta \in D$. 
	Consider $f \in N \cap \p$. 
	Define $g := f \cup \{ (\delta,\delta) \}$. 
	If $h \le g$, then since $\delta \in \dom(h)$, for all $\alpha < \delta$ in $\dom(h)$, 
	$h(\alpha) < \delta$. 
	It easily follows that $h \restrict \delta \in N \cap \p$, 
	and for all $h_1 \le h \restrict \delta$ in 
	$N \cap \p$, $h_1 \cup h$ is a condition below $h_1$ and $h$.
\end{proof}

\begin{proposition}
	$\textsf{PFA}(T^*)$ implies that for any sequence $\langle A_\alpha : \alpha < \omega_1 \rangle$ 
	of infinite subsets of $\omega_1$, there exists a club $E \subseteq \omega_1$ 
	such that for all $\alpha < \omega_1$, $A_\alpha \setminus E$ is infinite.
\end{proposition}

To prove the proposition, we apply $\textsf{PFA}(T^*)$ to the forcing poset 
$\textrm{CU}(\omega_1)$, which is strongly proper and hence $T^*$-proper. 
The argument is essentially the same as that of \cite[Theorem 3.4]{baumgartnerproper}, which drew 
the same conclusion from $\textsf{PFA}$.

It easily follows from the last proposition that 
\textsf{PFA}($T^*$) implies the failure of weak club guessing. 
But we can do better.

\begin{thm}
	$\textsf{PFA}(T^*)$ implies the failure of very weak club guessing.
	\end{thm}

\begin{proof}
	Fix a ladder system 
	$\vec L = \langle L_\alpha : \alpha \in \lim(\omega_1) \rangle$, and we will prove that 
	under \textsf{PFA}($T^*$), $\vec L$ is not a very weak club guessing sequence. 
	We may assume without loss of generality that for all $\alpha \in \lim(\omega_1)$, 
	$\alpha \setminus L_\alpha$ is cofinal in $\alpha$. 
	For if all ladder systems with this property are not very weak club guessing sequences, 
	then it is easy to argue that the same is true for all ladder systems.
	
	Define $\p$ as the forcing poset whose conditions are pairs $(f,x)$ such that 
	$f \in \textrm{CU}(\omega_1)$ and $x \subseteq \lim(\omega_1)$ is finite, 
	ordered by $(g,y) \le (f,x)$ if $x \subseteq y$, $f \subseteq g$, and for all 
	$\alpha \in x$, 
	$$
	(\dom(g) \setminus \dom(f)) \cap L_\alpha \cap \lim(\omega_1) = \emptyset.
	$$
	We will write $p = (f_p,x_p)$ for any $p \in \p$.

	Assuming that $\p$ preserves $\omega_1$, straightforward arguments prove that 
	$\p$ introduces a club $C \subseteq \omega_1$ such that for all $\alpha \in \lim(\omega_1)$, 
	$C \cap L_\alpha$ is finite, namely, the club of limit ordinals which belong to 
	the club set $\bigcup \{ \dom(f_p) : p \in \dot G_\p \}$. 
	So assuming that $\p$ is $T^*$-proper, a routine selection of dense sets shows that under 
	\textsf{PFA}($T^*$), $\vec L$ is not a very weak club guessing sequence. 

	It remains to prove that $\p$ is $T^*$-proper. 
	Fix a regular cardinal $\theta$ such that $T^*$ and $\p$ are members of $H(\theta)$. 
	Let $\mathcal X$ denote the set of countable elementary substructures of 
	$(H(\theta),\in,T^*,\p)$. 
	Observe that by Lemma 1.5, if $M \in \mathcal X$, then for all $x \in T^*_{M \cap \omega_1}$, 
	$x$ is $(M,T^*)$-generic.
	
	Let $N$ be a countable elementary substructure of $(H(\theta),\in,T^*,\p,\mathcal X)$. 
	Observe that $N \in \mathcal X$. 
	Let $\delta := N \cap \omega_1$. 	
	Consider $p \in N \cap \p$. 
	Define $q$ by $f_q := f_p \cup \{ (\delta,\delta) \}$ and $x_q := x_p$. 
	Then $q$ is a condition and $q \le p$. 
	We claim that $q$ is $(N,\p,T^*)$-generic. 
	So let $D \in N$ be a dense open subset of $\p$, $x \in T^*_\delta$, $\dot A \in N$ a $\p$-name 
	for a subset of $T^*$, and $R \in \{ <_{T^*}, <_{T^*}^- \}$.
	
	Let $r \le q$. 
	By extending further if necessary, we may assume without loss of generality 
	that $r \in D$ and $r$ decides whether or not $x \in \dot A$. 
	If $r$ forces that $x \notin \dot A$, then 
	replace $\dot A$ in what follows with the canonical $\p$-name for $T^*$. 
	Hence, without loss of generality, we may assume that $r$ forces that $x \in \dot A$. 
	Observe that since $\delta \in \dom(f_r)$, $f_r \restrict \delta$ maps into $\delta$, 
	and so is a member of $N$.

	Let $B_0$ denote the set of $z \in T^*$ for which there exists some $M \in \mathcal X$ such that 
	$D$ and $\dot A$ are members of $M$, 
	$M \cap \omega_1 = \h_{T^*}(z)$, and there exists some condition $s \in D$ such that 
	$f_s \restrict (M \cap \omega_1) = f_r \restrict \delta$, $M \cap \omega_1 \in \dom(f_s)$, 
	$x_s \cap M = x_r \cap \delta$, and $s$ forces that $z \in \dot A$. 
	Since all of the parameters mentioned in the definition of $B_0$ are members of $N$, 
	$B_0 \in N$ by elementarity. 
	Also, $x \in B_0$ as witnessed by $N$ and $r$.
	
	Since $x$ is $(N,T^*)$-generic, 
	we can fix $z <_{T^*} x$ which is in $B_0$. 
	Let $M$ and $s$ witness that $z \in B_0$. 
	Since $z \in N$, we may assume by elementarity that $M$ and $s$ are in $N$. 
	Let $\delta_0 := M \cap \omega_1$. 
	Note that since $\h_{T^*}(z) = M \cap \omega_1 = \delta_0$, 
	we have that $z$ is $(M,T^*)$-generic. 
	As $\delta_0 \in \dom(f_s)$, $f_s \restrict \delta_0 = f_r \restrict \delta$ maps 
	into $\delta_0$, and hence is a member of $M$.  
	(As a point of clarity, we do not claim that $r$ and $s$ are compatible; the argument 
	in what follows is more subtle.)
	
	Fix a limit ordinal $\xi < \delta_0$ which is large enough so that:
	\begin{itemize}
		\item $\dom(f_r) \cap \delta = \dom(f_s) \cap \delta_0$ is a subset of $\xi$;
		\item $\ran(f_r \restrict \delta) = \ran(f_s \restrict \delta_0) \subseteq \xi$;
		\item $x_r \cap \delta = x_s \cap \delta_0$ is a subset of $\xi$;
		\item for all $\beta \in x_r \setminus \delta$, $L_\beta \cap \delta_0 \subseteq \xi$;
		\item for all $\beta \in x_s \setminus (\delta_0 + 1)$, $L_\beta \cap \delta_0 \subseteq \xi$.
		\end{itemize}
	This is possible since $\delta_0$ is a limit of limit ordinals and all of the finitely many sets 
	mentioned above are finite subsets of $\delta_0$.
		
	Define $X$ as the set of $\beta$ satisfying:
	\begin{enumerate}
		\item[(a)] $\beta < \xi$ is a limit ordinal;
		\item[(b)] for all $\gamma \in \dom(f_r) \cap \delta$, it is not the case that 
		$\gamma \le \beta \le f_r(\gamma)$ (in particular, 
		$\beta \notin \dom(f_r) \cap \delta$);
		\item[(c)] there exists $\alpha \in x_r \setminus \delta$ 
		such that $\beta \in L_\alpha$.
		\end{enumerate}
	Observe that since $f_r \restrict \delta = f_s \restrict \delta_0$, (b) implies that 
	for all $\gamma \in \dom(f_s) \cap \delta_0$, it is not the case that 
	$\gamma \le \beta \le f_s(\gamma)$ (in particular, $\beta \notin \dom(f_s) \cap \delta_0$). 
	Note that $X$ is a finite subset of $\delta_0$, and hence is a member of $M$.
	
	For each $\beta \in X$, choose $\alpha_\beta < \beta$ large enough 
	so that:
	\begin{enumerate}
		\item $x_r \cap \beta$, $\dom(f_r) \cap \beta$, 
		$\ran(f_r) \cap \beta$, and $X \cap \beta$ are subsets of $\alpha_\beta$;
		\item for all $\alpha \in x_r \cup x_s$ 
		strictly larger than $\beta$, $L_\alpha \cap \beta \subseteq \alpha_\beta$;
		\item $\alpha_\beta \notin L_\beta$.
	\end{enumerate}
	This is possible since the finitely many 
	sets described in (1) and (2) are finite, 
	and by our assumption about the ladder system,  
	$\beta \setminus L_\beta$ is cofinal in $\beta$. 
	Observe that since 
	$f_r \restrict \delta = f_s \restrict \delta_0$ 
	and $x_r \cap \delta = x_s \cap \delta_0$, 
	(1) implies that 
	$x_s \cap \beta$, $\dom(f_s) \cap \beta$, and 
	and $\ran(f_s) \cap \beta$ are subsets of 
	$\alpha_\beta$. 
	Also note that by the definition of $X$, for all $\gamma \in \dom(f_r) \cap \beta = \dom(f_s) \cap \beta$, 
	$f_r(\gamma) = f_s(\gamma) < \beta$, and therefore $f_r(\gamma) = f_s(\gamma) < \alpha_\beta$.
	
	Define $v = (f_v,x_v)$ by letting $x_v := x_r \cap \delta$ 
	and 
	$$
	f_v := (f_{r} \restrict \delta) \cup \{ (\alpha_\beta,\beta) : \beta \in X \}.
	$$
	Then also $x_v = x_s \cap \delta_0$ and 
	$f_v = (f_{s} \restrict \delta_0) \cup \{ (\alpha_\beta,\beta) : \beta \in X \}$. 
	Clearly, $v$ is in $M$. 
	For $\beta \in X$, 
	for all $\gamma \in \dom(f_r) \cap \beta$, $f_r(\beta) < \alpha_\beta$, and also 
	if $\beta_1 < \beta_2$ are in $X$, then $\beta_1 < \alpha_{\beta_2}$. 
	It easily follows that $f_v \in \textrm{CU}(\omega_1)$, and hence $v \in \p$. 
	So $v \in M \cap \p$.
	
	We make two claims about $v$: (I) $v$ and $s$ are compatible, and (II) 
	any extension of $v$ in $M$ is compatible with $r$. 
	To see that $v$ and $s$ are compatible, let 
	$t := (f_v \cup f_s,x_v \cup x_s)$. 
	Since $f_s \restrict \delta_0 \subseteq f_v \in M$ and for all 
	$\gamma \in \dom(f_v) \cap \delta_0$, $f_v(\gamma) < \delta_0$, we have that 
	$f_v \cup f_s$ is in $\textrm{CU}(\omega_1)$. 
	Hence, $t$ is a condition. 
	To see that $t \le v$, we only need to show that 
	for all $\alpha \in x_v$, the set $\dom(f_t) \setminus \dom(f_v)$ is disjoint from 
	$L_\alpha \cap \lim(\omega_1)$. 
	But $x_v \subseteq \delta_0$, whereas 
	$\dom(f_t) \setminus \dom(f_v) \subseteq \omega_1 \setminus \delta_0$, so this is immediate.
		
	To show that $t \le s$, it is enough to show that for all 
	$\alpha \in x_s$, 
	$$
	(\dom(f_v) \setminus \dom(f_s)) \cap L_\alpha 
	\cap \lim(\omega_1) = \emptyset.
	$$
	So let $\alpha \in x_s$ and 
	let $\zeta$ be a limit ordinal in 
	$\dom(f_v) \setminus \dom(f_s)$, and we will prove that  
	$\zeta \notin L_\alpha$. 
	By the definition of $v$, 
	$\zeta = \alpha_\beta$ for some $\beta \in X$. 
	By property (2) in the definition of $\alpha_\beta$, 
	if $\beta < \alpha$ then 
	$L_\alpha \cap \beta \subseteq \alpha_\beta$, and thus 
	$\zeta \notin L_\alpha$. 
	So assume that $\alpha \le \beta$. 
	If $\alpha = \beta$, then by (3) in the definition 
	of $\alpha_\beta$, $\zeta \notin L_\alpha$. 
	Finally, if $\alpha < \beta$, then 
	by (1), $\alpha < \zeta$. 
	In any case, $\zeta \notin L_\alpha$. 
	This completes the proof that $v$ and $s$ 
	are compatible.
	
	Now we claim that for all $w \le v$ in $M \cap \p$, 
	$w$ and $r$ are compatible. 
	Fix $w \le v$ in $M \cap \p$. 
	Define $t := (f_w \cup f_r,x_w \cup x_r)$. 
	Since $f_r \restrict \delta \subseteq f_v \subseteq f_w$, 
	and for all $\gamma \in \dom(f_w)$, 
	$f_w(\gamma) < \delta$, we have that 
	$f_t \in \mathrm{CU}(\omega_1)$. 
	So $t$ is a condition. 
	As $x_w \subseteq \delta$ and 
	$\dom(f_t) \setminus \dom(f_w) \subseteq \omega_1 \setminus \delta$, 
	it easily follows that $t \le w$.
	
	It remains to show that $t \le r$. 
	It suffices to prove that for all $\alpha \in x_r$, 
	$\dom(f_t) \setminus \dom(f_r)$ is disjoint from $L_\alpha \cap \lim(\omega_1)$. 
	Fix $\alpha \in x_r$. 
	Let $\zeta \in \dom(f_t) \setminus \dom(f_r)$ be a 
	limit ordinal, and we will show that $\zeta \notin L_\alpha$. 
	Note that $\zeta \in \dom(f_w) \setminus \dom(f_r)$. 
	
	We split the argument into two cases. 
	First, assume that $\alpha < \delta$. 
	Then $\alpha \in x_r \cap \delta \subseteq x_v$. 
	Suppose for a contradiction that $\zeta \in L_\alpha$. 
	Then we cannot have $\zeta \in \dom(f_w) \setminus \dom(f_v)$, since $w \le v$. 
	Thus, $\zeta \in \dom(f_v)$. 
	As $\zeta \notin \dom(f_r)$, by the definition of 
	$v$ it must be the case that 
	$\zeta = \alpha_\beta$ for some $\beta \in X$. 
	Since $\zeta \in L_\alpha$, $\zeta < \alpha$. 
	If $\alpha < \beta$, then by (1) in the definition of $\alpha_\beta$, 
	we have that $\alpha < \alpha_\beta = \zeta$, which is false. 
	By (3) in the definition of $\alpha_\beta$, 
	if $\alpha = \beta$ then $\zeta \notin L_\alpha$. 
	Thus, $\beta < \alpha$. 
	But then by (2) in the definition of $\alpha_\beta$, 
	$L_\alpha \cap \beta \subseteq \alpha_\beta = \zeta$, so $\zeta \notin L_\alpha$. 
	This contradiction proves that $\zeta \notin L_\alpha$ 
	in the case that $\alpha < \delta$.
	
	Secondly, assume that $\delta \le \alpha$. 
	Suppose for a contradiction that $\zeta \in L_\alpha$. 
	We claim that $\zeta \in X$. 
	We know that $\zeta$ is a limit ordinal. 
	By the choice of $\xi$, the fact that 
	$\zeta \in L_\alpha$ implies that $\zeta < \xi$. 
	Let $\gamma \in \dom(f_r) \cap \delta$, and we will 
	prove that it is not the case that 
	$\gamma \le \zeta \le f_r(\gamma)$. 
	Suppose for a contradiction that 
	$\gamma \le \zeta \le f_r(\gamma)$. 
	Observe that $\gamma \in \dom(f_v)$ by the definition of $v$. 
	So we have that $\zeta \in \dom(f_w)$ and $\gamma \le \zeta \le f_w(\gamma)$. 
	Since $f_w \in \textrm{CU}(\omega_1)$, this implies that $\zeta = \gamma$. 
	So $\zeta \in \dom(f_r)$, which contradicts our assumption that $\zeta$ 
	is in $\dom(f_w) \setminus \dom(f_r)$. 
	Thus, indeed it is not the case that $\gamma \le \zeta \le f_r(\gamma)$. 
	Now by our current assumption, $\zeta \in L_\alpha$ where 
	$\alpha \in x_r \setminus \delta$. 
	This concludes the argument that $\zeta \in X$. 
	By the definition of $v$, 
	$f_v(\alpha_\zeta) = \zeta$, where 
	$\alpha_\zeta < \zeta$. 
	Since $w \le v$, 
	$f_w(\alpha_\zeta) = \zeta$. 
	But $\zeta \in \dom(f_w)$, so as 
	$f_w \in \textrm{CU}(\omega_1)$, 
	$f_w(\alpha_\zeta) < \zeta$, which is a contradiction. 
	This completes the proof that $\zeta \notin L_\alpha$, and hence 
	that $t \le w, r$.
	
	We can now complete the proof. 
	Let $B_1$ be the set of nodes $y \in T^*$ such that for some $w \le v$ in $D$, 
	$w$ forces that $y \in \dot A$. 
	Then $B_1 \in M$ by elementarity, and $z \in B_1$ as witnessed by any common lower bound of $v$ and $s$. 
	Since $z$ is $(M,T^*)$-generic, we can find $y \ R \ z$ in $B_1$. 
	Note that $y \ R \ z$ and $z <_{T^*} x$ imply that $y \ R \ x$. 
	As $y$ and $B_1$ are in $M$, by elementarity there is $w \in M \cap D$ such that 
	$w \le v$ and $w$ forces that $y \in \dot A$. 
	Since every extension of $v$ in $M \cap \p$ is compatible with $r$, 
	$w$ and $r$ are compatible. 
	Fix $t \le r, w$. 
	Then $y \ R \ x$ and $t$ forces that $y \in \dot A$. 
	Also, $t \le w$ and $w \in N \cap D$.
	\end{proof}

\section{Stationarily Knaster and cardinal characteristics}

We will prove in this section that \textsf{PFA}($T^*$) 
implies $\textsf{MA}_{\omega_1}$($\sigma$-linked), and in particular, that 
all of the standard cardinal characteristics of the continuum are greater than $\omega_1$. 
This result will factor through a property of forcing posets 
called stationarily Knaster. 
We will also prove that under \textsf{PFA}($T^*$), 
there is a Knaster forcing which is not 
stationarily Knaster, whereas under 
\textsf{MA}$_{\omega_1}$, every $\omega_1$-c.c.\! forcing 
is stationarily Knaster.

Let $\p$ be a forcing poset. 
A set $B \subseteq \p$ is \emph{linked} if any two members of $B$ are compatible, and 
is \emph{centered} if any finitely many members of $B$ have a lower bound. 
The forcing poset $\p$ is \emph{Knaster} if whenever 
$\{ p_i : i < \omega_1 \}$ is a family of conditions in $\p$, then there exists an 
uncountable set $X \subseteq \omega_1$ such that the collection 
$\{ p_i : i \in X \}$ is linked. 
Oftentimes in proving that a particular forcing is Knaster using pushing down arguments, 
a stronger form of the Knaster property can be derived.

\begin{definition}[\cite{SK}]
	A forcing poset $\p$ is \emph{stationarily Knaster} if whenever $\{ p_i : i \in S \}$ is a 
	family of conditions in $\p$, where $S \subseteq \omega_1$ is stationary, then there is a 
	stationary set $T \subseteq S$ such that the collection 
	$\{ p_i : i \in T \}$ is linked.
	\end{definition}

\begin{proposition}
	Let $\p$ be a stationarily Knaster forcing poset. 
	Then $\p$ is $T^*$-proper.
	\end{proposition}

\begin{proof}
	Fix a regular cardinal $\theta$ such that $T^*$ and $\p$ are members of $H(\theta)$. 
	Let $N$ be a countable elementary substructure of $(H(\theta),\in,T^*,\p)$. 
	Let $\delta := N \cap \omega_1$. 
	Since $\p$ is $\omega_1$-c.c., standard arguments show that 
	every condition in $\p$ is $(N,\p)$-generic. 
	We will show that every condition in $\p$ is $(N,\p,T^*)$-generic.
	
	Let $p \in \p$, $x \in T^*_{\delta}$, $\dot A \in N$ a $\p$-name for a subset of $T^*$, 
	and $R \in \{ <_{T^*}, <_{T^*}^- \}$. 
	Assume that $q \le p$ and $q$ forces that $x \in \dot A$. 
	We claim that there exists $y \ R \ x$ and $r \le q$ such that $r$ forces that $y \in \dot A$. 
	Suppose for a contradiction that the claim fails. 
	Then for all $y \ R \ x$, $q$ forces that $y \notin \dot A$.
	
	Define $B$ as the set of $z \in T^*$ such that for some $s \in \p$, $s$ forces that $z \in \dot A$, 
	but for all $y \ R \ z$, $s$ forces that $y \notin \dot A$. 
	Then $B \in N$ by elementarity, and $x \in B$ as witnessed by $q$. 
	Define $S := \{ \h_{T^*}(z) : z \in B \}$, which is in $N$. 
	Since $x \in B$, $\delta = \h_{T^*}(x)$ is in $S$. 
	But $\delta$ is a member of every club subset of $\omega_1$ which lies in $N$, so by 
	elementarity, $S$ is a stationary subset of $\omega_1$.
	
	For each $\alpha \in S$, fix a node $z_\alpha \in B$ with height $\alpha$ and 
	a condition $s_\alpha$ which witnesses that $z_\alpha \in B$. 
	Then $s_\alpha$ forces that $z_\alpha \in \dot A$, but for all $y \ R \ z_\alpha$, 
	$s_\alpha$ forces that $y \notin \dot A$. 
	We have that $\{ s_\alpha : \alpha \in S \} \subseteq \p$ and $S \subseteq \omega_1$ is stationary. 
	Since $\p$ is stationarily Knaster, we can find a stationary set $T \subseteq S$ 
	such that for all $\alpha < \beta$ in $T$, $s_\alpha$ and $s_\beta$ are compatible.
	
	Consider $\alpha < \beta$ in $T$. 
	Fix $r \le s_\alpha, s_\beta$. 
	Then $r$ forces that $z_\alpha$ and $z_\beta$ are in $\dot A$. 
	But $s_\beta$, and hence $r$, forces that any $y \ R \ z_\beta$ is not in $\dot A$. 
	So it must be the case that $z_\alpha$ is not $R$-below $z_\beta$.

	We consider two cases. 
	First, assume that $R$ equals $<_{T^*}$. 
	Then for all $\alpha < \beta$ in $T$, $z_\alpha \not <_{T^*} z_\beta$. 
	It follows that $z_\alpha$ and $z_\beta$ are incomparable in $T^*$. 
	Thus, $\{ z_\alpha : \alpha \in T \}$ is an antichain of $T^*$. 
	But the collection of heights of nodes in this antichain is exactly $T$, which 
	is stationary. 
	So $\{ z_\alpha : \alpha \in T \}$ is a stationary antichain, 
	which contradicts that $T^*$ has no stationary antichain.
			
	Secondly, assume that $R$ equals $<_{T^*}^-$. 
	Then for all $\alpha < \beta$ in $T$, it is not the case that 
	$z_\alpha <_{T^*}^- z_\beta$. 
	Since $\h_{T^*}(z_\alpha) = \alpha < \beta = \h_{T^*}(z_\beta)$, this means that 
	$z_\alpha <_{T^*} z_\beta$. 	
	Thus, $\{ z_\alpha : \alpha \in T \}$ is an uncountable chain of $T^*$, 
	which contradicts that $T^*$ is an $\omega_1$-Aronszajn tree. 
\end{proof}

\begin{corollary}
	$\textsf{PFA}(T^*)$ implies $\textsf{MA}_{\omega_1}(\textrm{stationarily Knaster})$.
	\end{corollary}

Recall that a forcing poset is \emph{$\sigma$-centered} if 
it is the union of countably many centered subsets, and is \emph{$\sigma$-linked} 
if it is the union of countably many linked subsets. 
Obviously, $\sigma$-centered implies $\sigma$-linked.

\begin{lemma}
	If $\p$ is a $\sigma$-linked forcing poset, then $\p$ is stationarily Knaster.
	\end{lemma}

\begin{proof}
	Let $\p = \bigcup \{ X_n : n < \omega \}$, where each $X_n$ is linked. 
	Consider a collection $\{ p_i : i \in S \}$ of conditions in $\p$, where $S \subseteq \omega_1$ is stationary. 
	If we define a function which maps each $\alpha \in S$ to the least $n < \omega$ 
	such that $p_\alpha \in X_n$, then there exists a stationary set $T \subseteq S$ 
	on which this function is constant. 
	Then for some $n < \omega$, $\{ p_i : i \in T \} \subseteq X_n$, and hence 
	$\{ p_i : i \in T \}$ is linked.
	\end{proof}

\begin{corollary}
	$\textsf{PFA}(T^*)$ implies \textsf{MA}$_{\omega_1}(\sigma$-\textrm{linked}).
	\end{corollary}

We can now conclude that \textsf{PFA}($T^*$) implies that all of the standard cardinal characteristics 
of the continuum are larger than $\omega_1$. 
Since the pseudo-intersection number $\mathfrak p$ is the smallest among such cardinal characteristics, 
it suffices to show that \textsf{PFA}($T^*$) implies that $\mathfrak p > \omega_1$. 
But $\mathfrak{p} > \omega_1$ is equivalent to 
the forcing axiom \textsf{MA}$_{\omega_1}$($\sigma$-centered) (\cite{bell}). 
And by Lemma 4.4, 
every $\sigma$-centered forcing poset is stationarily Knaster, and hence $T^*$-proper.

\begin{corollary}
	$\textsf{PFA}(T^*)$ implies that $\mathfrak p > \omega_1$.
\end{corollary}

It is a natural question whether the property of being Knaster is equivalent to being stationarily Knaster. 
Under \textsf{MA}$_{\omega_1}$, the answer is yes.

\begin{proposition}
	Assume $\textsf{MA}_{\omega_1}$. 
	Then every $\omega_1$-c.c.\! forcing poset is stationarily Knaster.
\end{proposition}

This easily follows from Lemma 4.4 and the well-known fact that 
under \textsf{MA}$_{\omega_1}$, every $\omega_1$-c.c.\! forcing poset 
of size $\omega_1$ is $\sigma$-centered.

The rest of this section is devoted to proving that under \textsf{PFA}($T^*$), 
the forcing poset consisting of finite antichains of $T^*$ ordered by reverse inclusion 
is Knaster but not stationarily Knaster. 

Recall the result of Baumgartner-Malitz-Reinhardt \cite{baumgartner} that 
if $T$ is a tree of height $\omega_1$ 
with no uncountable chains, and $\langle x_\alpha : \alpha < \omega_1 \rangle$ 
is a pairwise disjoint sequence of finite subsets of $T$, then there exist $\alpha < \beta$ 
such that every node in $x_\alpha$ is incomparable in $T$ with every node in $x_\beta$. 
This implies an apparently stronger conclusion.

\begin{lemma}
	Let $T$ be a tree of height $\omega_1$ with no uncountable chains, and let 
	$\langle x_\alpha : \alpha < \omega_1 \rangle$ 
	be a pairwise disjoint sequence of finite subsets of $T$. 
	Then there is a countably infinite set $A \subseteq \omega_1$ such that for all 
	$\alpha < \beta$ in $A$, every member of $x_\alpha$ is incomparable in $T$ 
	with every member of $x_\beta$.
	\end{lemma}

\begin{proof}
	Define a function $F : [\omega_1]^2 \to 2$ by $F(\alpha,\beta) = 0$ if there is a member of 
	$x_\alpha$ which is comparable in $T$ with a member of $x_\beta$, and 
	$F(\alpha,\beta) = 1$ otherwise. 
	By the Dushnik-Miller theorem $\omega_1 \to (\omega_1,\omega)^2$, 
	there exists either an uncountable subset of $\omega_1$ 
	all of whose pairs get value $0$, or a countably infinite subset of $\omega_1$ all of whose 
	pairs get value $1$. 
	Note that the first possibility contradicts the result of Baumgartner-Malitz-Reinhardt. 
	So the second possibility holds, which gives us the conclusion of the lemma.	
	\end{proof}

The next three results show 
that under \textsf{PFA}($T^*$), we can strengthen ``countably infinite'' 
in the last lemma to ``uncountable''.

\begin{lemma}
	Let $T$ be an $\omega_1$-tree. 
	Suppose that $\langle y_n : n < \omega \rangle$ is a sequence satisfying:
	\begin{enumerate}
		\item each $y_n$ is a finite subset of $T$;
		\item for all $m < n$, for all $x \in y_m$ and $z \in y_n$, 
		$\h_{T}(x) < \h_{T}(z)$ and $x$ and $z$ are incomparable in $T$.
		\end{enumerate}
	Let $y$ be a finite subset of $T$ such that for all 
	$n < \omega$, every node of $y_n$ has height less than the height of every 
	node in $y$. 
	Then there exists $n < \omega$ such that every member of $y_n$ is incomparable in 
	$T$ with every member of $y$.
	\end{lemma}

\begin{proof}
	Suppose for a contradiction that the conclusion of the lemma fails. 
	Define a function $f : \omega \to y$ by letting $f(n) \in y$ be some node which  
	is above some member of $y_n$ in $T$. 
	Then there is some $z \in y$ such that for infinitely many $n$, $f(n) = z$. 
	Fix $m < n$ such that $f(m) = f(n) = z$. 
	Pick $z_m \in y_m$ and $z_n \in y_n$ such that $z_m <_{T} z$ and $z_n <_{T} z$. 
	Since $T$ is a tree, it follows that $z_m <_{T} z_n$, which contradicts our assumption about $y_m$ and $y_n$.
	\end{proof}

\begin{thm}
	Let $T$ be a tree of height $\omega_1$ with no uncountable chains, and let 
	$\langle x_\alpha : \alpha < \omega_1 \rangle$ 
	be a pairwise disjoint sequence of finite subsets of $T$. 
	Then there exists a $T^*$-proper forcing poset which adds an 
	uncountable set $A \subseteq \omega_1$ such that for all 
	$\alpha < \beta$ in $A$, every member of $x_\alpha$ is incomparable in $T$ 
	with every member of $x_\beta$.
	\end{thm}

\begin{proof}
	By thinning out and relabelling if necessary, we may assume without loss of generality that 
	for all $\alpha < \beta$, every node of $x_\alpha$ has height at least $\alpha$ and 
	less than the height of any member of $x_\beta$.
	
	Let us say that a set $N$ is a \emph{suitable model} if it is a countable elementary 
	substructure of $H(\omega_2)$ which contains $T$ as an element. 
	Define a forcing poset $\p$ as follows. 
	Conditions in $\p$ are pairs $(a,X)$ satisfying:
	\begin{enumerate}
		\item $a \subseteq \omega_1$ is finite;
		\item for all $i < j$ in $a$, every node in $x_i$ is incomparable in $T$ 
		with every node in $x_j$;
		\item $X$ is a finite $\in$-chain of suitable models;
		\item for all $i < j$ in $a$, there is $N \in X$ such that $i < N \cap \omega_1 < j$;
		\item if $N \in X$ and $i \in a \setminus N$, then for any set $Z \subseteq \omega_1$ 
		in $N$ with $i \in Z$, there are uncountably many $j \in Z$ such that every member of $x_i$ 
		is incomparable in $T$ with every member of $x_j$;
		\item if $N \in X$ and $i \in a \setminus N$, then for any function $f : \omega_1 \to \omega_1$ 
		in $N$, $f(N \cap \omega_1) < i$;
		\item if $a \ne \emptyset$, then $X \ne \emptyset$ and there is some $K \in X$ 
		such that $K \cap \omega_1 < \min(a)$.
		\end{enumerate}
	Let $(b,Y) \le (a,X)$ if $a \subseteq b$ and $X \subseteq Y$.
	For any condition $p \in \p$, we will write $p = (a_p,X_p)$.

	We begin with the claim that for all $p \in \p$, 
	there are uncountably many $\xi < \omega_1$ such that for all 
	$i \in a_p$, every member of $x_i$ is incomparable in $T$ with every member of $x_\xi$. 
	If $a_p = \emptyset$, then the claim holds vacuously. 
	Otherwise fix $n < \omega$ such that $|a_p| = n+1$.
	
	Let $i^*$ be the largest member of $a_p$.
	By the definition of $\p$, there exists $N \in X$ such that 
	$a_p \setminus \{ i^* \} \subseteq N$ and $i^* \notin N$. 
	Let $Z$ be the set of $\xi < \omega_1$ such that for all $i \in a_p \setminus \{ i^* \}$, 
	every member of $x_i$ is incomparable in $T$ with every member of $x_\xi$ (for clarity, 
	if $n = 0$ then $Z = \omega_1$, which is fine). 
	Note that $Z \in N$ by elementarity, and $i^* \in Z$. 
	By the definition of $\p$, 
	there are uncountably many $\xi \in Z$ such that 
	every member of $x_{i^*}$ is incomparable in $T$ with every member of $x_\xi$, which 
	proves the claim.
	
	Assume for a moment that $\p$ is $T^*$-proper. 
	Let $\dot A$ be a $\p$-name for the set 
	$$
	\bigcup \{ a_p : p \in \dot G_\p \}.
	$$
	Then by the definition of $\p$, it is forced that 
	for all $i < j$ in $\dot A$, every node 
	in $x_i$ is incomparable in $T$ with every node in $x_j$. 
	
	We claim that $\dot A$ is forced to be uncountable, which completes 
	the proof of the theorem 
	under the assumption that $\p$ is $T^*$-proper. 
	To prove this, consider a condition $p \in \p$ and $\gamma < \omega_1$. 
	We will find $q \le p$ so that $a_q \setminus \gamma \ne \emptyset$. 
	Fix a suitable model $N$ such that $p$ and $\gamma$ 
	are members of $N$. 
	Then easily $(a_p,X_p \cup \{ N \})$ is a condition below $p$. 
	Let $Y$ be the set of $\xi < \omega_1$ such that for all $i \in a_p$, 
	every member of $x_i$ is incomparable in $T$ with every member of $x_\xi$. 
	Then $Y \in N$ by elementarity, and $Y$ is uncountable by the claim above.
	
	Let $Y'$ be the set of $\xi \in Y$ which satisfy the property that for 
	all $Z \subseteq \omega_1$ in $N$ with $\xi \in Z$, there are uncountably many 
	$j \in Z$ such that every member of $x_\xi$ is incomparable in $T$ 
	with every member of $x_j$. 
	We claim that $Y'$ is uncountable. 
	If not, then there exists $\beta < \omega_1$ such that for all 
	$\xi \in Y \setminus \beta$, there is $Z_\xi \in N$ with $\xi \in Z_\xi$ 
	and $\delta_\xi < \omega_1$ such that for all 
	$j \in Z_\xi \setminus \delta_\xi$, there is a member of $x_\xi$ which is comparable in $T$ 
	with some member of $x_j$. 
	Since $N$ is countable, we can fix $Z \in N$ such that 
	there are uncountably many $\xi \in Y \setminus \beta$ for which $Z_\xi = Z$. 
	Now we can construct by induction an increasing sequence of ordinals 
	$\langle \xi_i : i < \omega_1 \rangle$ in $Y \setminus \beta$ such that for all 
	$i < \omega_1$, $Z_{\xi_i} = Z$, and for all $i < j < \omega_1$, $\delta_{\xi_i} \le \xi_j$. 
	But now for all $i < j$, $\xi_j \in Z_{\xi_i} \setminus \delta_{\xi_i}$, 
	so there is a member of $x_{\xi_i}$ which is comparable 
	in $T$ with a member of $x_{\xi_j}$. 
	Thus, the sequence $\langle x_{\xi_i} : i < \omega_1 \rangle$ is a counterexample to the 
	result of Baumgartner-Malitz-Reinhart, which is a contradiction.
		
	Since $Y'$ is uncountable, we can fix 
	$\xi \in Y'$ which is larger than $N \cap \omega_1$, $\gamma$, and 
	$f(N \cap \omega_1)$ for all $f : \omega_1 \to \omega_1$ in $N$. 
	Then by the definition of $Y'$, it is easy to check that 
	$(a_p \cup \{ \xi \}, X_p \cup \{ N \} )$ is a condition below $p$. 
	This completes the proof that $\dot A$ is forced to be uncountable.
	
	It remains to prove that $\p$ is $T^*$-proper. 
	Fix a regular cardinal $\theta > \omega_2$ such that $T^*$, $T$, and $\p$ 
	are members of $H(\theta)$, and let 
	$N$ be a countable elementary substructure of $(H(\theta),\in,T^*,T,\p)$. 
	Let $\delta := N \cap \omega_1$. 
	Consider $p \in N \cap \p$. 
	Define $q := (a_p,X_p \cup \{ N \cap H(\omega_2) \})$. 
	Then $q$ is a condition and $q \le p$. 

	We claim that $q$ is $(N,\p,T^*)$-generic. 	
	Fix a dense set $D \in N$, $x \in T^*_\delta$, 
	a $\p$-name $\dot A \in N$ for a subset 
	of $T^*$, and a relation $R \in \{ <_{T^*}, <_{T^*}^- \}$. 
	Let $r \le q$. 
	By extending further if necessary, we may assume that $r \in D$, $r$ decides whether 
	or not $x \in \dot A$, and $a_r \setminus \delta \ne \emptyset$. 
	If $r$ forces that $x \notin \dot A$, then replace $\dot A$ in what follows by the 
	canonical $\p$-name for $T^*$. 
	Thus, we may assume without loss of generality that 
	$r$ forces that $x \in \dot A$. 
		
	Our goal is to find $t \le r$ which is below some member of $N \cap D$ and a node 
	$y \ R \ x$ such that $t$ forces that $y \in \dot A$. 
	Let $r_N := (a_r \cap N,X_r \cap N)$. 
	It is easy to check that $r_N$ is in $N \cap \p$ and $r \le r_N$. 
	Let $B_0$ denote the set of nodes $z \in T^*$ 
	for which there exists a condition $s \in D$ 
	such that:
	\begin{enumerate}
		\item $s \le r_N$;
		\item $a_s \cap \h_{T^*}(z) = a_r \cap \delta$;
		\item $s$ forces that $z \in \dot A$.
		\end{enumerate}
	Observe that $B_0 \in N$ by elementarity, and $x \in B_0$ as witnessed by $r$.
	
	For each $z \in B_0$, let $B_0(z)$ denote the 
	set of conditions $s \in D$ which witness that $z \in B_0$. 
	The function $z \mapsto B_0(z)$ is in $N$ by elementarity, and $r \in B_0(x)$. 
	Define a function $F :  B_0 \to \omega_1 + 1$ 
	by 
	$$
	F(z) := \sup \{ \min(a_s \setminus \h_{T^*}(z)) : s \in B_0(z) \}.
	$$
	Note that $F \in N$ by elementarity. 
	Define 
	$$
	X := \{ z \in B_0 : F(z) < \omega_1 \}.
	$$
	Define $H : \omega_1 \to \omega_1$ by 
	$$
	H(\alpha) := \sup \{ F(z) : z \in X, \ \h_{T^*}(z) = \alpha \}.
	$$
	Then $X$ and $H$ are in $N$ by elementarity.
	
	Property (6) in the definition of $r$ being a condition implies that 
	$H(\delta) < \min(a_r \setminus \delta)$. 
	In particular, $x$ cannot be in $X$. 
	For otherwise $F(x)$ would be in the set for which $H(\delta)$ is the supremum, 
	which implies that $\min(a_r \setminus \delta) \le F(x) \le H(\delta)$, giving a contradiction. 
	Therefore, $F(x) = \omega_1$.
	
	Let $B_1 := \{ z \in B_0 : F(z) = \omega_1 \}$. 
	Then $B_1 \in N$ by elementarity, and $x \in B_1$ as just proven. 
	Since $x$ is $(N,T^*)$-generic, there exists $y \ R \ x$ such that $y \in B_1$. 
	Let $\delta_0 := \h_{T^*}(y)$. 
	Since $F$ and $y$ are in $N$ and $F(y) = \omega_1$,  
	it is straightforward to construct in $N$ a sequence 
	$\langle s_i : i < \omega_1 \rangle$ of conditions in $B_0(y)$ such that 
	for all $i < j < \omega_1$, 
	$\max(a_{s_i}) < \min(a_{s_j} \setminus \delta_0)$.
	
	For each $i < \omega_1$, let 
	$y_i := \bigcup \{ x_\alpha : \alpha \in a_{s_i} \setminus \delta_0 \}$. 
	Then $\langle y_i: i < \omega_1 \rangle$ is a pairwise disjoint sequence of 
	finite subsets of $T$. 
	Note that for all $i < j$, every node in $y_i$ has height in $T$ less than the height of every node in $y_j$. 
	Apply Lemma 4.8 inside $N$ to find a set $z \subseteq \omega_1$ in $N$ with order type $\omega$ 
	such that for all $i < j$ in $z$, every member of $y_i$ is incomparable in $T$ with every member of $y_j$. 
	Define $y := \bigcup \{ x_i : i \in a_r \setminus \delta \}$. 
	Note that since $z \in N$, for all $i \in z$, every node in $y_i$ has height less than $\delta$, 
	whereas every node in $y$ has height at least $\delta$. 
	So we can apply Lemma 4.9 to $\langle y_i : i \in z \rangle$ and $y$ 
	to find $\beta \in z$ such that every node in $y_\beta$ is incomparable in $T$ with every node in $y$.
	
	Let $s := s_\beta$. 
	We claim that $r$ and $s$ are compatible. 
	Define $t := (a_r \cup a_{s},X_r \cup X_{s})$. 
	We will show that $t$ is in $\p$. 
	Then clearly $t \le r, s$. 
	It is easy to check that $t$ satisfies all of the properties of being in $\p$ except possibly (2).
	
	Consider $i < j$ in $a_t$, and we will prove that every node in $x_i$ is 
	incomparable in $T$ with every node in $x_j$. 
	If $i$ and $j$ are either both in $a_r$ or both in $a_{s}$, then 
	the desired conclusion is true since 
	$r$ and $s$ are conditions. 
	Otherwise $i \in a_{s} \setminus a_r$ and 
	$j \in a_r \setminus a_s$. 
	Since $s \in B_0(y)$, $a_s \cap \delta_0 = a_r \cap \delta$. 
	Thus, $i \in a_s \setminus \delta_0$ and $j \in a_r \setminus \delta$. 
	So $x_i \subseteq y_\beta$ and $x_j \subseteq y$. 
	By the choice of $\beta$, 
	every node of $x_i$ is incomparable in $T$ with every node of $x_j$.
	
	So we have that $t \in \p$ and $t \le r, s$. 
	In particular, $t$ is below a member of $N \cap D$, namely $s$. 
	Also $y \ R \ x$ and $s$ forces that $y \in \dot A$. 
	So $t \le r$ and $t$ forces that $y \in \dot A$, 
	which completes the proof that $q$ is $(N,\p,T^*)$-generic.
	\end{proof}

\begin{corollary}
	Assume $\textsf{PFA}(T^*)$. 
	Let $T$ be a tree of height $\omega_1$ with no uncountable chains, and let 
	$\langle x_\alpha : \alpha < \omega_1 \rangle$ be a 
	pairwise disjoint sequence of finite subsets of $T$. 
	Then there exists an uncountable set $A \subseteq \omega_1$ such that for all 
	$\alpha < \beta$ in $A$, every member of $x_\alpha$ is incomparable in $T$ with every member of $x_\beta$.
	\end{corollary}

\begin{proof}
	Let $\p$ be the $T^*$-proper forcing from Theorem 4.10, and let $\dot A$ be a $\p$-name for an 
	uncountable subset of $\omega_1$ such that for all $\alpha < \beta$ in $\dot A$, 
	every member of $x_\alpha$ is incomparable in $T$ with every member of $x_\beta$. 
	For each $\alpha < \omega_1$, let $D_\alpha$ be the dense set of conditions in $\p$ 
	which decide the $\alpha$-th member of $\dot A$. 
	Let $G$ be a filter on $\p$ which meets each dense set $D_\alpha$. 
	For each $\alpha$, let $\gamma_\alpha$ be the unique ordinal such that for some $p \in D_\alpha \cap G$, 
	$p$ forces that $\gamma_\alpha$ is the $\alpha$-th member of $\dot A$. 
	By the fact that $G$ is a filter, it is easy to check that for all $\alpha < \beta$, 
	every member of $x_{\gamma_\alpha}$ is incomparable in $T$ with every member of $x_{\gamma_\beta}$.
	\end{proof}

If we ignore $T^*$, then Theorem 4.10 shows more generally that there exists a proper forcing 
which adds a set $A$ as described in the theorem. 
Consequently, the conclusion of Corollary 4.11 follows from \textsf{PFA}.

\begin{corollary}
	$\textsf{PFA}(T^*)$ implies Suslin's hypothesis.
	\end{corollary}

\begin{proof}
	Let $T$ be an $\omega_1$-tree. 
	For each $\alpha < \omega_1$, choose a node $t_\alpha$ of $T$ with height $\alpha$, and 
	define $x_\alpha := \{ t_\alpha \}$. 
	Apply Corollary 4.11 to the collection $\{ x_\alpha : \alpha < \omega_1 \}$ 
	to find an uncountable set $A \subseteq \omega_1$ so that 
	for all $\alpha < \beta$ in $A$, $t_\alpha$ and $t_\beta$ are incomparable in $T$. 
	Then $\{ t_\alpha : \alpha \in A \}$ is an uncountable antichain of $T$. 
	So $T$ is not an $\omega_1$-Suslin tree.
	\end{proof}

\begin{thm}
	Assume $\textsf{PFA}(T^*)$. 
	Let $\p$ be the forcing poset whose conditions are finite antichains of $T^*$, 
	ordered by reverse inclusion. 
	Then $\p$ is Knaster but not stationarily Knaster.
	\end{thm}

\begin{proof}
	Note that conditions $x$ and $y$ of $\p$ are compatible iff $x \cup y$ 
	is an antichain of $T^*$. 
	For each $\alpha < \omega_1$, pick an element $t_\alpha$ of $T^*$ with height $\alpha$, and let 
	$p_\alpha := \{ t_\alpha \}$. 
	Then $p_\alpha$ and $p_\beta$ are compatible in $\p$ iff $t_\alpha$ and $t_\beta$ 
	are incomparable in $T^*$. 
	
	If $S \subseteq \omega_1$, then $\{ p_\alpha : \alpha \in S \}$ is a linked subset of $\p$ 
	iff for all $\alpha < \beta$ in $S$, $t_\alpha$ and $t_\beta$ are incomparable in $T^*$ iff 
	$\{ t_\alpha : \alpha \in S \}$ is an antichain of $T^*$. 
	Since $T^*$ has no stationary antichains, if $S$ is stationary then 
	$\{ p_\alpha : \alpha \in S \}$ is not linked. 
	It follows that $\p$ is not stationarily Knaster.

	To see that $\p$ is Knaster assuming \textsf{PFA}($T^*$), 
	let $\{ q_\alpha : \alpha < \omega_1 \}$ be a family of conditions in $\p$. 
	By the $\Delta$-system lemma, we may assume without loss of generality that there is a finite 
	set $r \subseteq T^*$ such that for all $\alpha < \beta < \omega_1$, 
	$q_\alpha \cap q_\beta = r$. 

	Let $x_\alpha := q_\alpha \setminus r$ for all $\alpha < \omega_1$. 
	Then $\langle x_\alpha : \alpha < \omega_1 \rangle$ is a pairwise disjoint 
	sequence of finite subsets of $T^*$. 	
	By Corollary 4.11, there exists an uncountable set $A \subseteq \omega_1$ such that for all $\alpha < \beta$ 
	in $A$, every member of $x_\alpha$ is incomparable in $T^*$ with every member of $x_\beta$. 
	It immediately follows that for all $\alpha < \beta$ in $A$, $q_\alpha \cup q_\beta$ is an antichain of $T^*$, 
	and hence a condition below $q_\alpha$ and $q_\beta$.
	\end{proof}

\section{Digression: more on stationarily Knaster}

In the previous section we proved that under \textsf{PFA}($T^*$), there exists a forcing poset 
which is Knaster but not stationarily Knaster. 
We now prove that the same conclusion follows from $\Diamond^*$. 
Afterwards, we introduce and explore the property of having stationary precaliber $\omega_1$.

Recall that $\Diamond^*$ is the statement that there exists a sequence 
$\langle \mathcal S_\alpha : \alpha < \omega_1 \rangle$ satisfying:
\begin{enumerate}
	\item each $\mathcal S_\alpha$ is a collection of countably many subsets of $\alpha$;
	\item for all $X \subseteq \omega_1$, there exists a club $D \subseteq \omega_1$ such that 
	for all $\alpha \in D$, $X \cap \alpha \in \mathcal S_\alpha$.
	\end{enumerate}

\begin{thm}
	$\Diamond^*$ implies that there exists a forcing poset which is Knaster but not 
	stationarily Knaster.
	\end{thm}

\begin{proof}
	Fix a sequence $\langle \mathcal F_\alpha : \alpha < \omega_1 \rangle$ satisfying:
	\begin{enumerate}
		\item each $\mathcal F_\alpha$ is a collection of countably many functions 
		from $\alpha$ to $\omega$;
		\item for any function $F : \omega_1 \to \omega$, 
		there exists a club $D \subseteq \omega_1$ such that for all $\alpha \in D$, 
		$F \restrict \alpha \in \mathcal F_\alpha$.
		\end{enumerate}
	Standard arguments show that $\Diamond^*$ implies the existence of such a sequence.
	
	For each limit ordinal $\alpha < \omega_1$, fix a cofinal subset $L_\alpha$ of $\alpha$ 
	with order type $\omega$. 
	Let $\alpha < \omega_1$ be a limit ordinal, and consider the countable 
	collection of functions $\{ f \restrict L_\alpha : f \in \mathcal F_\alpha \}$. 
	An easy diagonalization argument shows that there exists a function 
	$g_\alpha : L_\alpha \to \omega$ satisfying that for all $f \in \mathcal F_\alpha$, 
	there exists $\gamma < \alpha$ such that 
	for all $\beta \in L_\alpha \setminus \gamma$, $f(\beta) < g_\alpha(\beta)$. 
	In particular, for all $f \in \mathcal F_\alpha$, $g_\alpha \ne f \restrict L_\alpha$.
	
	We now consider the forcing poset from \cite{devlin} for adding a uniformisation of 
	the colored ladder system $\langle (L_\alpha,g_\alpha) : \alpha \in \lim(\omega_1) \rangle$. 
	Let $\p$ be the forcing poset whose conditions are pairs $(x,h)$ satisfying:
	\begin{enumerate}
		\item $x \subseteq \lim(\omega_1)$ is finite;
		\item $h : \bigcup \{ L_\alpha : \alpha \in x \} \to \omega$;
		\item for all $\alpha \in x$ there is $\gamma < \alpha$ such that for all 
		$\beta \in L_\alpha \setminus \gamma$, $h(\beta) = g_\alpha(\beta)$.
		\end{enumerate}	
	Let $(x_2,h_2) \le (x_1,h_1)$ if $x_1 \subseteq x_2$ and $h_1 \subseteq h_2$.
	
	We claim that $\p$ is Knaster but not stationarily Knaster. 
	To see that it is Knaster, we argue in a similar way as in \cite{devlin}. 
	Suppose that $\langle (x_i,h_i) : i < \omega_1 \rangle$ is a sequence of conditions in $\p$. 
	Applying the $\Delta$-system lemma and a standard thinning out argument, we can find 
	an uncountable set $Z \subseteq \omega_1$ such that for some set $r$, 
	for all $i < j$ in $Z$, $x_i \cap x_j = r$ and $\max(x_i) < \min(x_j \setminus r)$. 
	By relabeling the sequence if necessary, we may assume without loss of generality 
	that $Z = \omega_1$.
	
	For each $i < \omega_1$, choose a sequence $\langle \gamma^i_\alpha : \alpha \in r \rangle$ 
	such that for each $\alpha \in r$, $\gamma^i_\alpha < \alpha$ and 
	for all $\beta \in L_\alpha \setminus \gamma^i_\alpha$, 
	$h_i(\beta) = g_\alpha(\beta)$. 
	By a pressing-down argument, 
	we can find $Y_1 \subseteq \omega_1$ stationary so that for all $i < j$ in $Y_1$:
	\begin{enumerate}
		\item $\gamma^i_\alpha = \gamma^j_\alpha$ for all $\alpha \in r$;
		\item $h_i \restrict (L_\alpha \cap \gamma^i_\alpha) = 
		h_j \restrict (L_\alpha \cap \gamma^j_\alpha)$ for all $\alpha \in r$;
		\item $x_i \cap i = x_j \cap j = r$;
		\item $\max(x_i) < j$.
		\end{enumerate}
	Observe that (1) and (2) imply that for all $i < j$ in $Y_1$, 
	$h_i \restrict \bigcup \{ L_\alpha : \alpha \in r \} = 
	h_j \restrict \bigcup \{ L_\alpha : \alpha \in r \}$.
	
	For each $i \in Y_1$, let $i^+$ denote the least member of $Y_1$ strictly above $i$. 
	Then by construction, for each $i \in Y_1$, 
	$$
	\min(x_{i^+} \setminus r) > \max(x_i) \ge i.
	$$
	Consequently, the set $\bigcup \{ L_\alpha \cap i : \alpha \in x_{i^+} \setminus r \}$ 
	is a finite subset of $i$. 
	So we can find a stationary set $Y_2 \subseteq Y_1$ and a set $s$ 
	such that for all $i \in Y_2$, 
	$$
	\bigcup \{ L_\alpha \cap i : \alpha \in x_{i^+} \setminus r \} = s.
	$$	
	Now find $Y_3 \subseteq Y_2$ stationary such that for all $i < j$ in $Y_3$, 
	$i^+ < j$ and $h_{i^+} \restrict s = h_{j^+} \restrict s$.
	It is now easy to check that for all $i < j$ in $Y_3$, 
	$(x_{i^+} \cup x_{j^+},h_{i^+} \cup h_{j^+})$ is a lower bound of 
	$(x_{i^+},h_{i^+})$ and $(x_{j^+},h_{j^+})$. 
	This completes the proof that $\p$ is Knaster.
	
	Now we prove that $\p$ is not stationarily Knaster. 
	For each limit ordinal $\alpha < \omega_1$, 
	define $p_\alpha := (\{ \alpha \},g_\alpha)$. 
	Then obviously $p_\alpha \in \p$. 
	Suppose for a contradiction that there exists 
	a stationary set $T \subseteq \lim(\omega_1)$ 
	such that for all $i < j$ in $T$, $p_i$ and $p_j$ are compatible. 
	Then for all $i < j$ in $T$, 
	$(\{ i, j \}, g_i \cup g_j) \in \p$, and in particular, 
	$g_i \cup g_j$ is a function. 
	So for all $\gamma \in L_i \cap L_j$, $g_i(\gamma) = g_j(\gamma)$.
	
	Define a function $F : \omega_1 \to \omega$ so that for all $\gamma < \omega_1$, 
	$F(\gamma) = g_i(\gamma)$ if $i \in T$ and $\gamma \in L_i$, and if there is no such $i$, 
	then $F(\gamma) = 0$. 
	By the choice of $T$, $F$ is well-defined. 
	By the choice of the sequence 
	$\langle \mathcal F_\alpha : \alpha \in \lim(\omega_1) \rangle$, 
	there exists a club $D \subseteq \omega_1$ 
	such that for all $\alpha \in D$, $F \restrict \alpha \in \mathcal F_\alpha$. 
	Since $T$ is stationary, we can fix $\alpha \in T \cap D$. 
	Then $F \restrict \alpha = f$ for some $f \in \mathcal F_\alpha$. 
	So 
	$$
	f \restrict L_\alpha = F \restrict L_\alpha = g_\alpha
	$$
	But this contradicts the choice of $g_\alpha$.
	\end{proof}

Recall that a forcing poset $\p$ has \emph{precaliber $\omega_1$} if whenever 
$X \subseteq \p$ is uncountable, then there exists 
an uncountable set $Y \subseteq X$ which is centered. 
With the difference between Knaster and stationarily Knaster in mind, we introduce 
the following variation of this idea.

\begin{definition}
	A forcing poset $\p$ has \emph{stationary precaliber $\omega_1$} if whenever 
	$\{ p_i : i \in S \} \subseteq \p$, where $S \subseteq \omega_1$ is stationary, 
	then there exists a stationary set $T \subseteq \omega_1$ such that 
	$\{ p_i : i \in T \}$ is centered.
	\end{definition}

\begin{lemma}
	If $\p$ is a $\sigma$-centered forcing poset, then $\p$ has 
	stationary precaliber $\omega_1$.
\end{lemma}

\begin{proof}
	Similar to the proof of Lemma 4.4.
	\end{proof}

\begin{proposition}
	$\textsf{MA}_{\omega_1}$ implies that every $\omega_1$-c.c.\! forcing poset 
	has stationary precaliber $\omega_1$.
	\end{proposition}

\begin{proof}
	Similar to the proof of Proposition 4.7, but using 
	Lemma 5.3 instead of Lemma 4.4.
	\end{proof}

On the other hand, it is consistent that there exists a forcing poset which is stationarily 
Knaster but does not have stationary precaliber $\omega_1$.

\begin{proposition}
	Assume that $\mathfrak t = \omega_1$. 
	Then there exists a forcing poset which is stationarily Knaster but does not have 
	stationary precaliber $\omega_1$. 
	In particular, this is true if \textsf{CH} holds.
	\end{proposition}

\begin{proof}
	By \cite[Theorem 1.3]{TV}, there exists a forcing poset of size $\mathfrak t$ 
	which is $\sigma$-linked but has no centered subsets of size $\mathfrak t$. 
	Assuming $\mathfrak t = \omega_1$, it follows that there exists a forcing poset $\p$ of size 
	$\omega_1$ which is 
	$\sigma$-linked but has no centered subsets of size $\omega_1$.
	
	Since $\p$ is $\sigma$-linked, $\p$ is stationarily Knaster by Lemma 4.4. 
	Fix a sequence $\langle p_i : i < \omega_1 \rangle$ of distinct conditions in $\p$. 
	Then for any stationary set $T \subseteq \omega_1$, 
	$\{ p_i : i \in T \}$ is not centered, since $\p$ 
	contains no centered subsets of size $\omega_1$. 
	Thus, $\p$ does not have stationary precaliber $\omega_1$.
	\end{proof}

Since $\Diamond^*$ implies \textsf{CH}, the following corollary is immediate from the previous results.

\begin{corollary}
	Assume $\Diamond^*$. 
	Then there exists a forcing poset which is Knaster but not stationarily Knaster, and there exists 
	a forcing poset which is stationarily Knaster but does not have stationary precaliber $\omega_1$.
	\end{corollary}	

\section{$P$-ideal dichotomy}

In this section we will prove that \textsf{PFA}($T^*$) implies the $P$-ideal dichotomy. 
This will allow us to draw strong consequences from \textsf{PFA}($T^*$), such as the 
failure of square principles.

We recall the relevant definitions.

\begin{definition}
	Let $S$ be an uncountable set and $\mathcal I$ an ideal who members are countable subsets of $S$ 
	and which contains all finite subsets of $S$. 
	We say that $\mathcal I$ is a \emph{$P$-ideal} over $S$ 
	if whenever $\{ A_n : n < \omega \} \subseteq \mathcal I$, then there is $B \in \mathcal I$ 
	such that for all $n < \omega$, $A_n \subseteq^* B$.
\end{definition}

\begin{definition}
	The \emph{$P$-ideal dichotomy} is the statement that whenever $\mathcal I$ is a 
	$P$-ideal over an uncountable set $S$, 
	then either there exists 
	an uncountable set $T \subseteq S$ such that $[T]^\omega \subseteq \mathcal I$, or 
	$S = \bigcup \{ S_n : n < \omega \}$, where for each $n < \omega$, 
	$[S_n]^\omega \cap \mathcal I = \emptyset$.
\end{definition}

Fix a $P$-ideal $\mathcal I$ over an uncountable set $S$, and we will prove that 
under \textsf{PFA}($T^*$), one of the two alternatives of the $P$-ideal dichotomy holds for $\mathcal I$. 
For a set $X \subseteq S$, we will write $X \perp \mathcal I$ to mean that 
$[X]^\omega \cap \mathcal I = \emptyset$. 
Note that for any finite set $x \subseteq S$, $x \perp \mathcal I$. 
If $S$ is equal 
to a union $\bigcup \{ S_n : n < \omega \}$, where for each $n$, $S_n \perp \mathcal I$, 
then we are done. 
So assume not. 

We will prove that there exists a $T^*$-proper forcing poset which adds an uncountable 
set $T \subseteq S$ such that $[T]^\omega \subseteq \mathcal I$. 
There are several proper forcing posets in the literature for adding such a set. 
We will prove that the version from 
\cite{pid2} is $T^*$-proper.

Recall that a set $K \subseteq \mathcal I$ is \emph{cofinal} 
if for all $x \in \mathcal I$, there is $y \in K$ 
such that $x \subseteq^* y$. 
For a cofinal set $K \subseteq \mathcal I$ and $a \in \mathcal I$, define 
$K \restrict a := \{ b \in K : a \subseteq b \}$.

\begin{definition}
Define $\p$ as the forcing poset consisting of pairs $(a,X)$, where $a \in \mathcal I$ and 
$X$ is a countable collection of cofinal subsets of $\mathcal I$, ordered by 
$(b,Y) \le (a,X)$ if $a \subseteq b$, $X \subseteq Y$, and for all $K \in X$, the set 
$K \restrict (b \setminus a)$ is in $Y$.
\end{definition}

For any condition $p \in \p$, we will write $p = (a_p,X_p)$.

Observe that if $p \in \p$ and $K \subseteq \mathcal I$ is cofinal, then $(a_p,X_p \cup \{ K \})$ 
is a condition below $p$.

The next two lemmas follow by straightforward 
arguments. See \cite{pid2} for the details.

\begin{lemma}
Let $p \in \p$. 
For any countable set $S_0 \subseteq S$, there is $q \le p$ such that 
$a_q \setminus S_0 \ne \emptyset$.
	\end{lemma}

\begin{lemma}
	Suppose that $p \in \p$, $z \in S$, and $p$ forces that $z \in a_s$ for some $s \in \dot G_\p$. 
	Then $z \in a_p$. 
	\end{lemma}

We will prove below that the forcing poset 
$\p$ is $T^*$-proper and countably distributive. 
Assume for a moment that this is true. 
Observe that since $\p$ is countably distributive, $\p$ is still a $P$-ideal over $S$ in $V^{\p}$. 
Let $\dot T$ be a $\p$-name for the set 
$$
\bigcup \{ a_p : p \in \dot G_\p \}.
$$
Then Lemma 6.4 implies that $\p$ 
forces that $\dot T$ is uncountable, and 
Lemma 6.5 implies that $\p$ forces that 
$[\dot T]^\omega \subseteq \mathcal I$. 

The next two lemmas, which follow from arguments appearing in \cite{pid2}, 
will be used to show that $\p$ is proper.

\begin{lemma}
	Let $K \subseteq \mathcal I$ be cofinal and $x \in \mathcal I$. 
	Then there exists a finite set $x_0 \subseteq x$ such that 
	$K \restrict (x \setminus x_0)$ is cofinal.	
\end{lemma}

\begin{lemma}
	Let $\theta$ be a regular cardinal such that 
	$S$, $\mathcal I$, and $\p$ are members of $H(\theta)$. 
	Let $N$ be a countable elementary substructure of $(H(\theta),\in,S,\mathcal I,\p)$. 
	Let $p \in N \cap \p$, $D \in N$ a dense subset of $\p$, 
	and $d \in \mathcal I$ such that for all $a \in N \cap \mathcal I$, 
	$a \subseteq^* d$. 
	Then there exists $q \le p$ in $N \cap D$ such that 
	$a_q \setminus a_p \subseteq d$.
	\end{lemma}

The next two lemmas will help us to prove that $\p$ is $T^*$-proper.

\begin{lemma}
	Let $\theta$ be a regular cardinal such that 
	$S$, $\mathcal I$, and $\p$ are members of $H(\theta)$, 
	and $N$ a countable elementary substructure of $(H(\theta),\in,S,\mathcal I,\p)$. 
	Suppose that $d \in \mathcal I$ satisfies that $d \subseteq N$ 
	and for all $a \in N \cap \mathcal I$, $a \subseteq^* d$. 
	Let $p \in N \cap \p$, $\dot A \in N$ a $\p$-name for a subset of $T^*$, 
	and $y \in N \cap T^*$. 
	Then either:
	\begin{enumerate}
		\item there exists $q \le p$ in $N \cap \p$ such that $a_q \setminus a_p \subseteq d$ 
		and $q$ forces that $y \in \dot A$, or 
		\item there exists a cofinal set $J \subseteq \mathcal I$ in $N$ such that 
		$(a_p,X_p \cup \{ J \})$ forces that $y \notin \dot A$.
		\end{enumerate}
	\end{lemma}

\begin{proof}
	Assume that (1) fails, and we will prove (2). 
	So there does not exist $q \le p$ in $N \cap \p$ such that 
	$a_q \setminus a_p \subseteq d$ and $q$ forces that $y \in \dot A$. 
	Let $J$ be the set of $a \in \mathcal I$ such that whenever $q \le p$ and $a_q \setminus a_p \subseteq a$, 
	then $q$ does not force that $y \in \dot A$. 
	Observe that $J \in N$ by elementarity.

	We claim that $J$ is cofinal in $\mathcal I$. 
	To prove the claim, it suffices to show that $d \in J$. 
	For if this is true, then for all $b \in N \cap \mathcal I$, there is $c \in J$ (namely $c = d$) such that 
	$b \subseteq^* c$. 
	Hence, by elementarity $J$ is cofinal in $\mathcal I$. 
	Suppose for a contradiction that $d \notin J$. 
	Then there exists $t \le p$ such that $a_t \setminus a_{p} \subseteq d$ 
	and $t$ forces that $y \in \dot A$. 
	Since $d \subseteq N$, $a_t \setminus a_{p}$ is a finite subset of $N$ and hence is a member of $N$. 
	By elementarity, there exists $q \in N \cap \p$ such that $q \le p$, 
	$a_q \setminus a_{p} = a_t \setminus a_{p}$, and $q$ forces that $y \in \dot A$. 
	Then $a_q \setminus a_{p} \subseteq d$. 
	So $q$ witnesses that (1) holds, which is a contradiction.
	
	It remains to show that $(a_p,X_p \cup \{ J \})$ forces that $y \not \in \dot A$. 
	If not, then there exists $q \le (a_p,X_p \cup \{ J \})$ 
	which forces that $y \in \dot A$. 
	By the definition of the order on $\p$, 
	$J \restrict (a_q \setminus a_p)$ is in $X_q$, and therefore is cofinal and hence non-empty. 
	Fix $a \in J$ such that $a_q \setminus a_p \subseteq a$. 
	Then $q \le p$, $a_q \setminus a_p \subseteq a$, and $q$ forces that $y \in \dot A$, 
	contradicting the fact that $a \in J$.
	\end{proof}

\begin{lemma}
	Let $\theta$ be a regular cardinal such that 
	$S$, $\mathcal I$, and $\p$ are members of $H(\theta)$, 
	and $N$ a countable elementary substructure of $(H(\theta),\in,S,\mathcal I,\p)$. 
	Let $\delta := N \cap \omega_1$. 
	Assume that $d \in \mathcal I$ satisfies that $d \subseteq N$ 
	and for all $a \in N \cap \mathcal I$, $a \subseteq^* d$. 
	Suppose that $x \in T^*_\delta$,  
	$\dot A \in N$ is a $\p$-name for a subset of $T^*$, 
	$R \in \{ <_{T^*}, <_{T^*}^- \}$, and $p \in N \cap \p$. 
	Then either:
	\begin{enumerate}
		\item there exists $y \ R \ x$ and $q \le p$ in $N \cap \p$ such that 
		$a_q \setminus a_p \subseteq d$ and $q$ forces that $y \in \dot A$, or
		\item there exists a cofinal set $J \subseteq \mathcal I$ such that 
		$(a_p,X_p \cup \{ J \})$ forces that $x \notin \dot A$.
	\end{enumerate}
	\end{lemma}

\begin{proof}
	If (1) fails, then by Lemma 6.8 we have that 
	for all $y \ R \ x$, 
	there exists a cofinal set 
	$J \subseteq \mathcal I$ such that 
	$(a_p,X_p \cup \{ J \})$ 
	forces that $y \notin \dot A$. 
	Let $B$ be the set of $z \in T^*$ 
	for which there exists a cofinal set $J \subseteq \mathcal I$ 
	such that $(a_p,X_p \cup \{ J \})$ forces that $z \notin \dot A$. 
	Then $B \in N$ by elementarity, and for all $y \ R \ x$, $y \in B$. 
	By Lemma 1.4, it follows that $x \in B$. 
	So there exists a cofinal set 
	$J \subseteq \mathcal I$ such that 
	$(a_p,X_p \cup \{ J \})$ forces that $x \notin \dot A$.
	\end{proof}

\begin{proposition}
The forcing poset $\p$ is $T^*$-proper and countably distributive.
\end{proposition}

\begin{proof}
	Fix a regular cardinal $\theta$ such that $T^*$, $S$, $\mathcal I$, and $\p$ are members of $H(\theta)$. 
	Let $N$ be a countable elementary substructure of $H(\theta)$ which contains these parameters. 
	Let $\delta := N \cap \omega_1$. 
	Consider $p \in N \cap \p$. 
	We will find $q \le p$ which is $(N,\p,T^*)$-generic. 
	In fact, $q$ will have the property that it is a member of every dense open subset of 
	$\p$ which lies in $N$. 
	Standard arguments show that this implies that $\p$ is countably distributive.
	
	Since $\mathcal I$ is a $P$-ideal and $N$ is countable, 
	we can fix a set $d \in \mathcal I$ such that for all $a \in N \cap \mathcal I$, 
	$a \subseteq^* d$. 
	By intersecting $d$ with $N$ if necessary, we may also assume that $d \subseteq N$. 

	Let $\langle D_n : n < \omega \rangle$ enumerate all dense open subsets of $\p$ in $N$. 
	Let $\langle (x_n,\dot A_n) : n < \omega \rangle$ list all pairs $(x,\dot A)$, where 
	$x \in T^*_{\delta}$ and $\dot A \in N$ is a $\p$-name for a subset of $T^*$. 
	Let $\langle (i_n,K_n) : n < \omega \rangle$ enumerate all pairs $(i,K)$ such that 
	$i < \omega$ and $K$ is a cofinal subset of $\mathcal I$ in $N$.
		
	We will define by induction on $n < \omega$ objects $p_n$, $d_n$, and a function $F$ whose domain 
	is a subset of $\omega$. 
	For bookkeeping, fix a surjection $f : \omega \to 3 \times \omega$ such that for each $(i,m) \in 3 \times \omega$, 
	$f^{-1}(i,m)$ is infinite. 
	We will maintain the following properties:
	\begin{enumerate}
		\item $p_0 = p$ and $d_0 = d$;
		\item for all $n < \omega$, $p_n \in N \cap \p$ and $p_{n+1} \le p_n$;
		\item for all $n < \omega$, $d_n \in \mathcal I$, 
		for all $a \in N \cap \mathcal I$, $a \subseteq^* d_n$, 
		and $d_{n+1} \subseteq d_n$;
		\item for all $n < \omega$, $a_{p_{n+1}} \setminus a_{p_n} \subseteq d_{n+1}$.  
		\end{enumerate}
	
	Let $p_0 := p$ and $d_0 := d$. 
	Assume that $n < \omega$ and $p_n$, $d_n$, and $F \restrict n$ are defined. 
	Let $f(n) = (i,m)$. 
	We split the construction into the three cases depending on whether $i$ is $0$, $1$, or $2$.
	
	Suppose that $i = 0$. 
	First, we specify that $n \notin \dom(F)$. 
	Secondly, in order to define $d_{n+1}$, 
	we consider the pair $(i_m,K_m)$. 
	We only care if $i_m \le n$ and $K_m \in X_{p_{i_m}}$. 
	If not, then let $d_{n+1} := d_n$. 
	Otherwise, since $p_{n} \le p_{i_m}$, 
	$K_m \restrict (a_{p_n} \setminus a_{p_{i_m}}) \in X_{p_n}$, 
	and in particular is cofinal. 
	By Lemma 6.6, we can find $d_{n+1} \subseteq d_n$ which 
	is equal to $d_n$ minus finitely many elements such that 
	$$
	(K_m \restrict (a_{p_n} \setminus a_{p_{i_m}})) 
	\restrict d_{n+1}
	$$
	is cofinal in $\mathcal I$. 
	By the choice of $d_{n+1}$, clearly for all $a \in N \cap \mathcal I$, 
	$a \subseteq^* d_{n+1}$. 
	Thirdly, apply Lemma 6.7 to fix 
	$p_{n+1} \le p_n$ in $D_m$ such that 
	$a_{p_{n+1}} \setminus a_{p_n} \subseteq d_{n+1}$.

	Suppose that $i = 1$. 
	Consider the pair $(x_m,\dot A_m)$. 
	Then $x_m$ is $(N,T^*)$-generic by Lemma 1.5. 
	There are two possibilities. 
	If there exists $y <_{T^*} x_m$ and $v \le p_n$ in $N \cap \p$ 
	such that $a_{v} \setminus a_{p_{n}} \subseteq d_n$ and $v$ forces that 
	$y \in \dot A_m$, then let $F(n)$ be equal such a $y$ and let $p_{n+1}$ 
	be such a $v$. 
	Also, define $d_{n+1} := d_n$.
	
	If not, then by Lemma 6.9 there exists a cofinal set $J \subseteq \mathcal I$ such that 
	$(a_{p_n},X_{p_n} \cup \{ J \})$ forces that $x_m \notin \dot A_m$. 
	Define $F(n)$ to be such a $J$. 
	(For clarity, $J$ is not necessarily in $N$.) 
	Now apply Lemma 6.6 to find $d_{n+1} \subseteq d_n$ which is equal to $d_n$ minus finitely many elements 
	such that $F(n) \restrict d_{n+1}$ is cofinal. 
	Also, let $p_{n+1} := p_n$.
	
	Suppose that $i = 2$. 
	This case is exactly the same as the case $i = 1$, except replacing the relation $<_{T^*}$ 
	with $<_{T^*}^-$.
	
	This completes the construction. 
	Let $X := \bigcup \{ X_{p_n} : n < \omega \}$. 
	Define $q = (a_q,X_q)$ by $a_q := \bigcup \{ a_{p_n} : n < \omega \}$ and 
	$$
	X_q := X \cup \{ K \restrict (a_q \setminus a_{p_i}) : i < \omega, \ K \in X_{p_i} \}.
	$$ 
	We claim that $q$ is a condition below each $p_n$.
	
	Observe that by construction, 
	for all $n < \omega$, 
	$a_q \setminus a_{p_n} \subseteq d_{n+1}$. 
	In particular, $a_q = a_{p_0} \cup (a_{q} \setminus a_{p_0}) \subseteq a_{p_0} \cup d_1$, which implies 
	that $a_q$ is in $\mathcal I$.
	
	Clearly $X$ is a countable collection of cofinal 
	subsets of $\mathcal I$. 
	Consider $i < \omega$ and $K \in X_{p_i}$, and we will 
	show that $K \restrict (a_q \setminus a_{p_i})$ is cofinal. 
	Fix $m$ such that $(i_m,K_m) = (i,K)$. 
	Let $n \ge i$ be such that $f(n) = (0,m)$. 
	Then $i_m = i \le n$ and $K_m = K \in X_{p_{i_m}}$. 
	By construction, 
	$$
	(K_m \restrict (a_{p_n} \setminus a_{p_{i_m}})) 
	\restrict d_{n+1}
	$$
	is cofinal in $\mathcal I$. 
	But we claim that 
	$K \restrict (a_q \setminus a_{p_i})$ contains this 
	cofinal set, and hence is itself cofinal. 
	To see this, assume that $a \in K_m$ contains 
	$a_{p_n} \setminus a_{p_{i_m}}$ and $d_{n+1}$. 
	Then $a \in K_m = K$. 
	We need to show that $a_q \setminus a_{p_i} \subseteq a$. 
	Since $a_{p_i} \subseteq a_{p_n} \subseteq a_q$, 
	we have that 
	$$
	a_q \setminus a_{p_{i}} = 
	(a_q \setminus a_{p_n}) \cup 
	(a_{p_{n}} \setminus a_{p_{i}}).
	$$
	And both of the sets in this union are subsets of $a$, 
	by the choice of $a$ and since 
	$a_{q} \setminus a_{p_n} \subseteq d_{n+1}$.

	Thus, $q$ is a condition, and easily by definition, $q \le p_n$ for all $n < \omega$.	
	By construction, for any dense open subset $D$ of $\p$ in $N$, $q \in D$. 
	So $q$ is $(N,\p)$-generic.

	In order to find an $(N,\p,T^*)$-generic condition, we need to extend $q$ further. 
	But before we do this, let us prove a claim. 
	Consider $n \in \dom(F)$ such that $F(n) \subseteq \mathcal I$ is cofinal. 
	We claim that $F := F(n) \restrict (a_q \setminus a_{p_n})$ is cofinal. 
	Note that by construction, we are in the case that $f(n) = (i,m)$, where $i \in \{ 1, 2 \}$, 
	$(a_{p_n},X_{p_n} \cup \{ F(n) \})$ forces that $x_m \notin \dot A_m$, 
	and $F(n) \restrict d_{n+1}$ is cofinal. 
	But since $a_q \setminus a_{p_n} \subseteq d_{n+1}$, we have that 
	$$
	F(n) \restrict d_{n+1} \subseteq F(n) \restrict (a_q \setminus a_{p_n}) = F.
	$$
	Thus, $F$ contains a cofinal set, and hence is itself cofinal.
	
	Define 
	$$
	X_1 := \{ F(n) : n \in \dom(F), \ F(n) \subseteq \mathcal I \ 
	\textrm{is cofinal} \}
	$$
	and 
	$$
	X_2 := \{ F(n) \restrict (a_q \setminus a_{p_n}) : n \in \dom(F), \ 
	F(n) \subseteq \mathcal I \ \textrm{is cofinal} \}.
	$$
	Note that by the claim of the previous paragraph, every member of $X_2$ 
	is cofinal. 
	Define 
	$$
	r := (a_q,X_q \cup X_1 \cup X_2).
	$$
	Then clearly $r \in \p$ and $r \le q$. 
	In particular, $r$ is $(N,\p)$-generic.

	It remains to show that $r$ is $(N,\p,T^*)$-generic. 
	So let $x \in T^*_\delta$, $\dot A \in N$ a $\p$-name 
	for a subset of $T^*$, and $R \in \{ <_{T^*}, <_{T^*}^- \}$. 
	Let $i$ be $1$ if $R$ equals $<_{T^*}$ and $2$ if $R$ equals $<_{T^*}^-$. 
	Fix $m$ such that $(x_m,\dot A_m) = (x,\dot A)$. 
	Fix $n < \omega$ such that $f(n) = (i,m)$.
	
	Recall that at stage $n$ we considered two possibilities. 
	The first is that there exists $y \ R \ x_m$ and $v \le p_n$ in $N \cap \p$ 
	such that $v$ forces that $y \in \dot A_m$ 
	and $a_v \setminus a_{p_{n}} \subseteq d_n$. 
	In that case, $F(n) \ R \ x_m$ and $p_{n+1}$ forces that $F(n) \in \dot A_m$. 
	Thus, $F(n) \ R \ x$ and $r$ forces that $F(n) \in \dot A$, and we are done.

	The second possibility is that there exists a cofinal set, which we defined as $F(n)$, 
	such that $(a_{p_n},X_{p_n} \cup \{ F(n) \})$ forces that $x_m \notin \dot A_m$. 
	By the definition of $r$, $F(n) \in X_r$ and 
	$F(n) \restrict (a_r \setminus a_{p_n}) = 
	F(n) \restrict (a_q \setminus a_{p_n}) \in X_r$. 
	It follows that $r \le (a_{p_n},X_{p_n} \cup \{ F(n) \})$. 
	Thus, $r$ forces that $x \notin \dot A$, which concludes the proof.
\end{proof}

The next theorem now follows by a 
straightforward argument.

\begin{thm}
$\textsf{PFA}(T^*)$ implies the $P$-ideal dichotomy.
\end{thm}

\begin{corollary}
	$\textsf{PFA}(T^*)$ implies that for all regular cardinals $\lambda > \omega_1$, 
	$\Box(\lambda)$ fails. 
	In particular, it implies that for all cardinals $\kappa \ge \omega_1$, 
	$\Box_{\kappa}$ fails.
\end{corollary}

\begin{corollary}
	$\textsf{PFA}(T^*)$ implies the singular cardinal hypothesis.
	\end{corollary}

\begin{corollary}
	$\textsf{PFA}(T^*)$ implies that for all regular uncountable cardinals $\kappa$ and $\lambda$, 
	if there exists a $(\kappa,\lambda^*)$-gap in $P(\omega)$, then $\kappa = \lambda = \omega_1$.
	\end{corollary}

These statements are immediate, since the conclusions of the corollaries are consequences of 
the $P$-ideal dichotomy (\cite{pid2}, \cite{viale}).

It is well-known that the $P$-ideal dichotomy combined with $\mathfrak p > \omega_1$ 
has many consequences which do not follow from the $P$-ideal dichotomy alone. 
Since \textsf{PFA}($T^*$) implies $\mathfrak p > \omega_1$ by Corollary 4.6, we can derive 
additional conclusions. 
We mention two important examples (see \cite{part}).

\begin{corollary}
	$\textsf{PFA}(T^*)$ implies that for all $\alpha < \omega_1$, 
	$\omega_1 \to (\omega_1,\alpha)^2$.
	\end{corollary}

\begin{corollary}
	$\textsf{PFA}(T^*)$ implies the non-existence of $S$-spaces.
	\end{corollary}

In contrast, the existence of a Suslin tree implies the existence of a compact 
$L$-space and an $S$-space (\cite{rudin}, \cite[Chapter 5]{todorbook}). 
This displays contrasting pictures provided by the forcing axioms $\textsf{PFA}(T^*)$ and 
$\textsf{PFA}(S)$.

Observe that our results prove that the $P$-ideal dichotomy together with $\mathfrak p > \omega_1$ 
does not imply that all $\omega_1$-Aronszajn trees are special (or even that all $\omega_1$-Aronszajn trees have 
a stationary antichain). 
A similar but more restrictive result was proven earlier in \cite{mildenberger}, 
where it was shown that an ideal dichotomy for 
$\omega_1$-generated ideals together with $\mathfrak p > \omega_1$ is consistent with a non-special Aronszajn tree.

\bigskip

We close the paper with some topics for future research. 
We currently do not know whether some other notable consequences of \textsf{PFA} not addressed in this paper 
follow from \textsf{PFA}($T^*$). 
These include the open coloring axiom, the mapping reflection principle, measuring, and the non-existence of 
weak Kurepa trees. 

An important part of Todorcevic's theory of \textsf{PFA}($S$) is the model \textsf{PFA}($S$)$[S]$, which is obtained 
by forcing over a model of \textsf{PFA}($S$) with the coherent Suslin tree $S$. 
In contrast, \textsf{PFA}($T^*$) implies that every $\omega_1$-Aronszajn tree is special on cofinally many levels 
(see the argument from \cite[Lemma 4.6, Chapter IX]{shelah}). 
It follows that forcing with $T^*$ collapses $\omega_1$. 
We do not know, however, whether generic extensions of models of \textsf{PFA}($T^*$) by other forcings associated 
with $T^*$ could be interesting, such as the forcing for specializing $T^*$ with finite appoximations.

The theory of \textsf{PFA}($T^*$) could possibly be relevant to the issues of chain conditions and Martin's 
axiom which motivated \textsf{PFA}($S$). 
For example, it is open whether Martin's axiom follows from the statement that all $\omega_1$-c.c.\! forcings are 
Knaster, or from the statement that the product of $\omega_1$-c.c.\! forcings is $\omega_1$-c.c.

Finally, we mention a problem suggested to the author by Justin Moore and the referee: 
Does $\textsf{PFA}(T^*)$ imply that 
all normal Aronszajn trees with no stationary antichains are club isomorphic? 
This would imply that the forcing axiom $\textsf{PFA}(T^*)$ is independent of the choice of $T^*$ and 
thus to some extent canonical.

\bibliographystyle{plain}
\bibliography{paper33}

\end{document}